\UseRawInputEncoding
\documentclass[12pt, reqno]{amsart}
\usepackage{amssymb,latexsym,amsmath,amscd,amsfonts}
\usepackage{latexsym}
\usepackage[mathscr]{eucal}
\usepackage{bm}
\usepackage{mathptmx}

9.5in \numberwithin{equation}{section}

\newcommand{\beg}{\begin{equation}}
\newcommand{\eeg}{\end{equation}}
\newcommand{\ben}{\begin{eqnarray*}}
	\newcommand{\een}{\end{eqnarray*}}

\newtheorem{thm}{Theorem}[section]

\newtheorem{lem}[thm]{Lemma}

\newtheorem{prop}[thm]{Proposition}
\numberwithin{equation}{section} 
\theoremstyle{definition}
\newtheorem{defn}[thm]{Definition}
\newtheorem{rem}[thm]{Remark}

\newcommand{\HS}{\mathcal H}
\newcommand{\C}{\mathbb{C}}

\newcommand{\D}{\mathbb{D}}
\newcommand{\T}{\mathbb{T}}
\newcommand{\ft}{\mathcal F_O}

\newcommand{\gn}{\mathbb{G}_n}

\newcommand{\ov}{\overline}

\newcommand{\la}{\langle}
\newcommand{\ra}{\rangle}

\begin{document}
\title[Distinguished varieties in the symmetried polydisc]
{Distinguished varieties in a family of domains associated with spectral interpolation and operator theory}

\author[Sourav Pal]{Sourav Pal}
%\dedicatory{\textup{(A tribute to Professor Jim Agler and Professor John E.
%M\raise.45ex\hbox{c}Carthy)}}
\address[Sourav Pal]{Mathematics Department, Indian Institute of Technology Bombay, Powai, Mumbai - 400076, India.}
\email{sourav@math.iitb.ac.in, souravmaths@gmail.com}



\thanks{The author is supported by the Seed Grant of IIT Bombay, the CPDA of the Govt. of India and the MATRICS Award of SERB, (Award No. MTR/2019/001010) of DST, India.}

\begin{abstract}

We characterize all distinguished varieties in the symmetrized polydisc $\mathbb G_n \; (n\geq 2)$ and thus generalize the work [\textit{J. Funct. Anal.}, 266 (2014), 5779 -- 5800] on $\mathbb G_2$ by the author and Shalit. We show that a distinguished variety $\Lambda$ in $\mathbb G_n$ is a part of an algebraic curve, which is a set-theoretic complete intersection, and that $\Lambda$ can be represented by the Taylor joint spectrum of $n-1$ commuting scalar matrices satisfying certain conditions. An $n$-tuple of commuting Hilbert space operators $(S_1, \dots ,S_{n-1},P)$ for which $\Gamma_n=\overline{\mathbb G_n}$ is a spectral set is called a $\Gamma_n$-contraction. To every $\Gamma_n$-contraction $(S_1, \dots ,S_{n-1},P)$ there is a unique operator tuple $(F_1, \dots , F_{n-1})$, called the $\mathcal F_O$-tuple of $(S_1, \dots ,S_{n-1},P)$, satisfying
\[
S_i-S_{n-i}^*P=D_PF_iD_P \,,\quad i=1, \dots ,n-1.
\]
We produce concrete functional model for the pure isometric-operator tuples associated with $\Gamma_n$ and by an application of that model we establish that the $\Gamma_n$-contractions $(S_1, \dots ,S_{n-1},P)$ and $(S_1^*, \dots , S_{n-1}^*,P^*)$ admit normal $\partial \overline{ \Lambda}_{\Sigma}-$dilations for a unique distinguished variety $\Lambda_{\Sigma}$ in $\gn$, when $\Lambda_{\Sigma}$ is determined by the $\mathcal F_O$-tuple of $(S_1, \dots ,S_{n-1}, P)$. Further, we show that the dilation of $(S_1^*, \dots ,S_{n-1}^*,P^*)$ is minimal and acts on the minimal unitary dilation space of $P^*$. Also, we show interplay between the distinguished varieties in $\mathbb G_2$ and $\mathbb G_{3}$.

\end{abstract}

\maketitle

%\tableofcontents

\section{Introduction} \label{sec:01}

\noindent We denote by $\mathbb C,\, \mathbb D $ and $\mathbb T$ the space of complex numbers, the unit disc and the unit circle in the complex plane with center at the origin respectively. All operators considered are bounded linear operators defined on complex Hilbert spaces. A contraction is an operator with norm not greater than
one. Throughout the paper we shall consider the symmetrized $n$-disc for $n\geq 2$ only.\\

The \textit{symmetrized} $n$-\textit{disc} or simply the \textit{symmetrized polydisc} $\mathbb G_n$, which is defined by
\[
\mathbb G_n =\left\{ \left(\sum_{1\leq i\leq n} z_i,\sum_{1\leq
i<j\leq n}z_iz_j,\dots, \prod_{i=1}^n z_i \right): \,|z_i|< 1,
i=1,\dots,n \right \},
\]
is family of polynomially convex domains that naturally arise in the Spectral Pick-Nevanlinna interpolation problem. If $\mathcal M_n(\mathbb C)$ is the space of $n\times n$ complex matrices and if $\mathcal B_1$ is its spectral unit ball, then $A\in \mathcal B_1$ (that is, the spectral radius $r(A)<1$) if and only if $\pi_n(\sigma(A)) \in \mathbb G_n$, where $\sigma(A)$ is the spectrum of $A$ and $\pi_n:\mathbb C^n \rightarrow \mathbb C^n$ is the symmetrization map whose components are the $n$ symmetric polynomials, i.e.
\[
\pi_n(z)=\left(\sum_{1\leq i\leq n} z_i,\sum_{1\leq
i<j\leq n}z_iz_j,\dots, \prod_{i=1}^n z_i \right).
\]
Evidently $\mathbb G_1$ is the unit disc $\D$ and for $n\geq 2$, $\mathbb G_n= \pi_n(\D^n)$. The most appealing fact about $\mathbb G_n$ is that apart from the derogatory matrices (for which it is not yet known), the $n\times n$ spectral Nevanlinna-Pick problem is equivalent to a similar interpolation problem of $\mathbb G_n$ (see \cite{ay-ieot}, Theorem 2.1). Note that a bounded domain like $\mathbb G_n$, which has complex-dimension $n$, is much easier to deal with than a norm-unbounded $n^2$-dimensional object like $\mathcal B_1$. The symmetrized polydisc has attracted considerable attention in past two decades because of its rich function theory, beautiful complex geometry and appealing operator theory (e.g. see \cite{costara1, edi-zwo, ay-jot, tirtha-sourav, BSR, Bisai-Pal1, Bisai-Pal2} and the references there in).\\

The main aim of this paper is to characterize all distinguished varieties (which we define below) in $\mathbb G_n$ for $n\geq 2$ and to study rational dilation for operators associated with them. The distinguished varieties in the bidisc $\D^2$ was introduced in \cite{AM05} by Agler and M\raise.45ex\hbox{c}Carthy, where they showed that for any pair of commuting contractive matrices $(T_1,T_2)$ of same order having no unimodular eigenvalues, there is a distinguished variety
\begin{equation}\label{eqn:Intro-02}
V=\{ (z,w)\in \mathbb D^2\,:\, \det (\Psi(z)-wI)=0 \}
\end{equation}
in $\D^2$ such that von-Neumann's inequality holds on $V\cap \D^2$, i.e.
\begin{equation}\label{eqn:Intro-01}
\| f(T_1,T_2) \|\leq \sup_{(z_1,z_2)\in V\cap \D^2} |f(z_1,z_2)|,
\end{equation}
holds for any polynomial $f\in \mathbb C[z_1,z_2]$. Note that here $\Psi$ is matrix-valued rational function on the unit disc $\mathbb D$ which is unitary on the unit circle $\mathbb T$. The seminal work of Agler and M\raise.45ex\hbox{c}Carthy was continued further in \cite{A-K-M}. In \cite{pal-shalit}, the author of this article and Shalit found an analogue of Agler and M\raise.45ex\hbox{c}Carthy's results for the symmetrized bidisc
\[
\mathbb G_2 = \{ (z_1+z_2,z_1z_2):\,, |z_i|<1, \, i=1,2 \}.
\]
Also, there are interesting recent developments on this topic, e.g. \cite{Bh-Ku-Sau, Da-Ku-Sa}.\\

\noindent Recall that for any positive integer $n$, $\mathbb A_{\mathbb C}^n$ stands for the \textit{affine} $n$-\textit{space} $\mathbb C^n$ over $\mathbb C$. If $A \subseteq \mathbb C[z_1,\dots,z_n]$, then the \textit{zero set} of $A$ is defined to be the set of common zeros of all elements of $A$, i.e.
\[
Z(A)=\{Q\in\mathbb A_{\mathbb C}^n\,:\, p(Q)=0 \text{ for all } p \in A   \}.
\]

\begin{defn}
A subset $S$ of $\mathbb A_{\mathbb C}^n$ is called an \textit{affine algebraic set} or simply an \textit{algebraic set} if there is a subset $A$ of $\mathbb C[z_1,\dots,z_n]$ such that $S=Z(A)$. An irreducible algebraic set is called an \textit{affine algebraic variety} or simply an \textit{algebraic variety}.
\end{defn}

\begin{defn}
A \textit{distinguished set} $S$ in a domain $\Omega\subset \mathbb C^n$ is the intersection of $\Omega$ with an algebraic set $V\subset \mathbb A^n_{\mathbb C}$ such that the complex-dimension of $V$ is greater than zero and that $V$ exits the domain $\Omega$ through the distinguished boundary $b\overline{\Omega}$ of $\Omega$, that is,
\[
\dim V >0, \quad S = \Omega \cap V \quad \text{ and } \quad \overline{S}\cap \partial \overline{\Omega}= \overline{S} \cap b\overline{\Omega},
\]
$\partial \overline{\Omega}$ being the topological boundary of $\Omega$. A \textit{distinguished variety} in $\Omega$ is an irreducible distinguished set. Also, we denote by $\overline{S}$ and $\partial \ov{S}$, the sets $V\cap \overline{\Omega}$ and $V\cap b\ov{\Omega}$ respectively.
\end{defn}
Thus, a distinguished set is nothing but a union of distinguished varieties and needless to mention that the distinguished sets in $\mathbb D^2$ and $\mathbb G_2$ were studied in \cite{AM05} and \cite{pal-shalit} respectively, not only the distinguished varieties. We now focus on the operator tuples that are naturally associated with the symmetrized polydisc.
\begin{defn}
A commuting $n$-tuple of Hilbert space operators $(S_1,\dots$, $S_{n-1},P)$ for which the closed symmetrized polydisc $\Gamma_n \; (=\ov{\gn})$ is a spectral set is called a $\Gamma_n$-\textit{contraction}, that is, $(S_1, \dots ,S_{n-1},P)$ is a $\Gamma_n$-contraction if the Taylor joint spectrum $\sigma_T(S_1, \dots, S_{n-1},P) \subseteq \Gamma_n$ and von-Neumann's inequality
\[
\|g(S_1, \dots ,S_{n-1},P)\| \leq \sup_{\textbf{z} \in \Gamma_n}\; |g(\textbf{z})|
\]
holds for every rational function $g=p/q$, where $p,q \in \C[z_1, \dots ,z_n]$ with $q$ does not have any zero in $\Gamma_n$.
\end{defn}
It is evident from the definition that if $(S_1,\dots, S_{n-1},P)$ is a $\Gamma_n$-contraction, then so is its adjoint $(S_1^*,\dots, S_{n-1}^*,P)^*$ and that $P$ is a contraction. It was shown in \cite{sourav6} that to every $\Gamma_n$-contraction $(S_1, \dots ,S_{n-1},P)$ acting on $\HS$, there is a unique operator tuple $(F_1, \dots ,F_{n-1})$ acting on $\ov{Ran}\, D_P=\ov{Ran}\, (I-P^*P)^{\frac{1}{2}}$, namely the \textit{fundamental operator tuple} or simply the $\ft$-\textit{tuple} of $(S_1, \dots ,S_{n-1},P)$, satisfying
\[
S_i-S_{n-i}^*P=D_PF_iD_P \quad \text{for } \; i=1, \dots ,n-1.
\]
In Theorems \ref{thm:DVchar} \& \ref{thm:DVchar-1} of this paper, we characterize all distinguished varieties in the symmetrized polydisc and represent them as the Taylor joint spectrum of certain commuting complex square matrices of same order. The Taylor joint spectrum of a finitely many commuting square matrices is just the set of their joint-eigenvalues. We prove that every distinguished variety in $\mathbb G_n$ is a part of an affine algebraic curve, which itself is also a set-theoretic complete intersection. In Theorem \ref{model1}, we obtain a concrete functional model for pure $\Gamma_n$-isometries. We apply this functional model in Theorem \ref{thm:VN} to show that a $\Gamma_n$-contraction $(S_1, \dots ,S_{n-1}, P)$ and its adjoint $(S_1^*, \dots , S_{n-1}^*,P^*)$ admit  normal $\partial \ov{\Lambda}_{\Sigma}-$dilation for a unique distinguished variety $\Lambda_{\Sigma}$ in $\gn$, if the $\ft$-tuple $(F_1, \dots ,F_{n-1})$ defines $\Lambda_{\Sigma}$. As a consequence of a celebrated theorem of Arveson (see Theorem \ref{Arveson}) it follows that the $\Gamma_n$-contractions $(S_1, \dots ,S_{n-1},P)$ and $(S_1^*, \dots ,S_{n-1}^*,P^*)$ satisfy von Neumann's inequality on an algebraic curve $\Lambda_{\Sigma}$ lying in $\gn$ for any holomorphic matricial polynomial. In the same section, we prove that every distinguished variety in $\mathbb G_n$ is polynomially convex. Also, we show by explicit construction that every distinguished variety in $\mathbb G_3$ gives rise to a distinguished variety in $\mathbb G_2$. In Section \ref{sec:02}, we discuss in brief the notion of Taylor joint spectrum and recollect from the literature a few results on the symmetrized polydisc which will be used in sequel.\\

\noindent \textbf{Acknowledgement.} The author is grateful to Indranil Biswas, Sudhir Ghorpade, Souvik Goswami and Jugal K. Verma for stimulating conversations on general algebraic geometry. Also, the author thanks Prof. J. E McCarthy profusely for being kind enough to read the paper and make fruitful comments.

%\noindent \textbf{Note.} The present article is an updated and enlarged version of a part of author's unpublished note \cite{sourav15}. 
 
\vspace{0.1cm}

\section{Preliminaries and a brief literature on the symmetrized polydisc}\label{sec:02}

\vspace{0.4cm}

\noindent The notion of Taylor joint spectrum of commuting tuples of operators and distinguished boundary of a domain play pivotal role in this paper. For this reason we provide a brief discussion on these topics.
\subsection{Taylor joint spectrum} 
Let $\Lambda$ be the exterior algebra on $n$ generators
$e_1,...e_n$, with identity $e_0\equiv 1$. $\Lambda$ is the
algebra of forms in $e_1,...e_n$ with complex coefficients,
subject to the collapsing property $e_ie_j+e_je_i=0$ ($1\leq i,j
\leq n$). Let $E_i: \Lambda \rightarrow \Lambda$ denote the
creation operator, given by $E_i \xi = e_i \xi $ ($\xi \in
\Lambda, 1 \leq i \leq n$).
 If we declare $ \{ e_{i_1}... e_{i_k} : 1 \leq i_1 < ... < i_k \leq n \}$ to be an
 orthonormal basis, the exterior algebra $\Lambda$ becomes a Hilbert space,
 admitting an orthogonal decomposition $\Lambda = \oplus_{k=1} ^n \Lambda^k$
 where $\dim \Lambda ^k = {n \choose k}$. Thus, each $\xi \in \Lambda$ admits
 a unique orthogonal decomposition
 $ \xi = e_i \xi' + \xi''$, where $\xi'$ and $\xi ''$ have no $e_i$
contribution. It then follows that that  $ E_i ^{*} \xi = \xi' $,
and we have that each $E_i$ is a partial isometry, satisfying
$E_i^*E_j+E_jE_i^*=\delta_{i,j}$. Let $\mathcal X$ be a normed
space, let $\underline{T}=(T_1,...,T_n)$ be a commuting $n$-tuple
of bounded operators on $\mathcal X$ and set $\Lambda(\mathcal
X)=\mathcal X\otimes_{\mathbb{C}} \Lambda$. We define
$D_{\underline T}: \Lambda (\mathcal X) \rightarrow \Lambda
(\mathcal X)$ by

\[
D_{\underline T} = \sum_{i=1}^n T_i \otimes E_i .
\]

Then it is easy to see $D_{\underline T}^2=0$, so $Ran
D_{\underline T} \subset Ker D_{\underline T}$. The commuting
$n$-tuple is said to be \textit{non-singular} on $\mathcal X$ if
$Ran D_{\underline T}=Ker D_{\underline T}$.
\begin{defn}
The Taylor joint spectrum of ${\underline T}$ on $\mathcal X$ is
the set
\[
\sigma_T({\underline T},\mathcal X) = \{
\lambda=(\lambda_1,...,\lambda_n)\in \mathbb{C}^n : {\underline
T}-\lambda \text{ is singular} \}.
\]
\end{defn}
\begin{rem}
The decomposition $\Lambda=\oplus_{k=1}^n \Lambda^k$ gives rise to
a cochain complex $K({\underline T},\mathcal X)$, known as the
Koszul complex associated to ${\underline T}$ on $\mathcal X$, as
follows:
\[
K({\underline T},\mathcal X):0 \rightarrow \Lambda^0(\mathcal
X)\xrightarrow{D_{\underline T}^0}... \xrightarrow{D_{\underline
T}^{n-1}} \Lambda^n(\mathcal X) \rightarrow 0 ,
\]
where $D_{\underline T}^{k}$ denotes the restriction of
$D_{\underline T}$ to the subspace $\Lambda^k(\mathcal X)$. Thus,
\[
\sigma_T({\underline T},\mathcal X) = \{ \lambda\in \mathbb{C}^n :
K({\underline T}-\lambda ,\mathcal X)\text{ is not exact} \}.
\]
\end{rem}
The following theorem is well-known.

\begin{thm}\label{thm:Joint Spectrum}
For any finitely many commuting complex square matrices of same size $B_1,\dots ,B_n$, the Taylor joint spectrum $\sigma_T(B_1,\dots ,B_n)$ is the set of joint eigenvalues of $B_1,\dots, B_n$.
\end{thm}

\noindent For a further reading on Taylor joint spectrum an
interested reader is referred to Curto's nice survey article \cite{curto} or Taylor's original works \cite{Taylor, Taylor1}.

\subsection{The distinguished boundary}

Let $X\subset \mathbb C^n$ be a compact set and let $\mathcal A(X)$ be an algebra of continuous complex-valued
functions on $X$ which separates the points of $X$. A
\textit{boundary} for $\mathcal A(X)$ is a subset $\delta
X$ of $X$ such that every function in $\mathcal A(X)$ attains its
maximum modulus on $\delta X$. It follows from the theory of
uniform algebras that the intersection of all closed boundaries of
$X$ is also a boundary for $\mathcal A(X)$ (see Theorem 9.1 in
\cite{wermer}). This smallest boundary is called the
\textit{Shilov boundary} for $\mathcal A(X)$.
When $\mathcal A(X)$ is the uniform closure of the algebra of rational functions which are continuous on $X$, the Shilov boundary for $\mathcal A(X)$ is called the \textit{distinguished boundary} of $X$ and is denoted by $bX$.

\subsection{A brief literature on the symmetrized polydisc} \label{sec:03}

\vspace{0.3cm}

\noindent Recall from the literature (e.g. \cite{BSR}) that the closed symmetrized polydisc $\Gamma_n$ and its distinguished boundary $b\Gamma_n$ are the following subsets of $\mathbb C^n$:
\begin{align*}
\Gamma_n & = \left\{ \left(\sum_{1\leq i\leq n} z_i,\sum_{1\leq
i<j\leq n}z_iz_j,\dots, \prod_{i=1}^n z_i \right): \,|z_i|\leq 1,
i=1,\dots,n \right \}, \\
b\Gamma_n & = \left\{ \left(\sum_{1\leq i\leq n} z_i,\sum_{1\leq
i<j\leq n}z_iz_j,\dots, \prod_{i=1}^n z_i \right): \,|z_i|= 1,
i=1,\dots,n \right \}.
\end{align*}
In this section, we recall from the lierature (\cite{costara1, BSR, sourav3}) a few results on complex geometry and operator theory on $\mathbb G_n$. We begin with a theorem that characterizes a point in $\Gamma_n$ and $\mathbb G_n$.

\begin{thm}[\cite{costara1}, Theorems 3.6 \& 3.7]\label{char-G}
   For $x=(s_1,\dots, s_{n-1},p)\in \mathbb C^n$ the following are
   equivalent:
   \begin{enumerate}
        \item[$(1)$] $ x\in \Gamma_n  ($ respectively $\in\mathbb G_n)$ ;
        \item[$(2)$] $ |p|\leq 1$,  $($ respectively $<1)$ and there
        exists
        $ (c_1,\dots, c_{n-1}) \in \Gamma_{n-1}  \;($ respectively $\in\mathbb G_{n-1})$ such that
        \[
        s_i = c_i + \bar{c}_{n-i}p \,,\, \text{ for } i=1,\dots, n-1.
        \]
   \end{enumerate}
\end{thm}

\noindent The next theorem provides a set of
characterizations for the points in the distinguished boundary $b\Gamma_n$.

\begin{thm}[\cite{BSR}, Theorem 2.4]\label{thm:DB}
For $(s_1,\dots,s_{n-1},p)\in \mathbb C^n$ the following are
equivalent:
\begin{enumerate}
\item $(s_1,\dots, s_{n-1},p)\in b\Gamma_n$\,; \item $(s_1,\dots,
s_{n-1},p)\in\Gamma_n$ and $|p|=1$\,; \item $|p|=1$,
$s_i=\bar{s}_{n-i}p$ for $i=1,\dots,n-1$ and
\[
\left( \frac{n-1}{n}s_1,\frac{n-2}{n}s_2,\dots,
\frac{n-({n-1})}{n}s_{n-1} \right)\in \Gamma_{n-1}\,;
\]
\item $|p|=1$ and there exist $(c_1,\dots,c_{n-1})\in\Gamma_{n-1}$
such that
\[
s_i=c_i+\bar{c}_{n-i}p \;\; \text{ for } \;\; i=1,\dots,n-1.
\]
\end{enumerate}
\end{thm}

\noindent Recall that a $\Gamma_n$-contraction is an $n$-tuple of commuting Hilbert space operators $(S_1, \dots ,S_{n-1},P)$ that has $\Gamma_n$ as a spectral set.

\begin{defn}
A $\Gamma_n$-contraction $(S_1,\dots, S_{n-1},P)$ is called \textit{pure} or $C._0$ if $P$ is a pure contraction, that is, ${P^*}^n \rightarrow 0$ strongly as $n\rightarrow \infty$.
\end{defn}

\noindent The unitaries, isometries and co-isometries are special classes of contractions. There are natural analogues for these classes in the literature of $\Gamma_n$-contractions, e.g. \cite{ay-jot, Bisai-Pal1, BSR, sourav3}.
\begin{defn}
Let $S_1,\dots, S_{n-1},P$ be commuting operators on $\mathcal H$.
Then $(S_1,\dots, S_{n-1},P)$ is called
\begin{itemize}
\item[(i)] a $\Gamma_n$-\textit{unitary} if $S_1,\dots, S_{n-1},P$
are normal operators and $\sigma_T
(S_1,\dots, S_{n-1},P) \subseteq b\Gamma_n$ ;

\item[(ii)] a $\Gamma_n$-isometry if there exists a Hilbert space
$\mathcal K \supseteq \mathcal H$ and a $\Gamma_n$-unitary
$(T_1,\dots,T_{n-1},U)$ on $\mathcal K$ such that $\mathcal H$ is
a joint invariant subspace of $S_1,\dots, S_{n-1},P$ and that
\[
(T_1|_{\mathcal H},\dots, T_{n-1}|_{\mathcal H},U|_{\mathcal
H})=(S_1,\dots, S_{n-1},P) ;\]

\item[(iii)] a $\Gamma_n$-co-isometry if the adjoint
$(S_1^*,\dots, S_{n-1}^*,P^*)$ is a $\Gamma_n$-isometry.
\end{itemize}

\end{defn}

The following theorem provides a set of characterizations for the $\Gamma_n$-unitaries.

\begin{thm}[\cite{BSR}, Theorem 4.2]\label{G-unitary}
Let $(S_1,\dots, S_{n-1}, P)$ be a commuting tuple of bounded
operators. Then the following are equivalent.

\begin{enumerate}

\item $(S_1,\dots,S_{n-1},P)$ is a $\Gamma_n$-unitary,

\item $P$ is a unitary and $(S_1,\dots,S_{n-1},P)$ is a
$\Gamma_n$-contraction,

\item $P$ is a unitary,
$(\frac{n-1}{n}S_1,\frac{n-2}{n}S_2,\dots,\frac{1}{n}S_{n-1})$ is
a $\Gamma_{n-1}$-contraction and $S_i = S_{n-i}^* P$ for
$i=1,\dots,n-1$.
\end{enumerate}
\end{thm}

\noindent The following result is a structure theorem for
$\Gamma_n$-isometries.

\begin{thm}[\cite{BSR}, Theorem 4.12]\label{G-isometry}
Let $S_1,\dots,S_{n-1},P$ be commuting operators on a Hilbert
space $\mathcal H$. Then the following are equivalent:
\begin{enumerate}

\item $(S_1,\dots,S_{n-1},P)$ is a $\Gamma_n$-isometry ; \item $P$
is isometry, $S_i=S_{n-i}^*P$ for each $i=1,\dots,n-1$ and
$(\frac{n-1}{n}S_1,\frac{n-2}{n}S_2,\dots,\frac{1}{n}S_{n-1})$ is
a $\Gamma_{n-1}$-contraction ; \item $($Wold-Decomposition$)$:
there is an orthogonal decomposition $\mathcal H=\mathcal H_1
\oplus \mathcal H_2$ into common invariant subspaces of
$S_1,\dots,S_{n-1}, P$ such that $(S_i|_{\mathcal
H_1},\dots,S_{n-1}|_{\mathcal H_1},P|_{\mathcal H_1})$ is a
$\Gamma_n$-unitary and the counter part $(S_1|_{\mathcal
H_2},\dots,S_{n-1}|_{\mathcal H_2},P|_{\mathcal H_2})$ is a pure
$\Gamma_n$-isometry ; \item $(S_1,\dots,S_{n-1},P)$ is a
$\Gamma_n$-contraction and $P$ is an isometry.

\end{enumerate}

\end{thm}

\vspace{0.3cm}

\section{Distinguished varieties in the symmetrized polydisc}\label{sec:04}

\vspace{0.3cm}

\noindent We begin by recalling a few definitions and results from the literature of algebraic geometry.

\begin{defn}
 An \textit{affine algebraic curve} or simply an \textit{algebraic curve} in the affine space $\mathbb A^n_{\mathbb C}$ is an algebraic set that has dimension one as a complex manifold.
\end{defn}

\begin{defn}
An affine algebraic set or variety $V$ of dimension $k \,(\leq n)$ in $\mathbb A_{\mathbb C}^n$ is called a \textit{set-theoretic complete intersection} or simply a \textit{complete intersection} if $V$ is defined by $n-k$ independent polynomials in $\mathbb C[z_1.\dots, z_n]$. Thus an algebraic curve $C$ in $\mathbb A^n_{\mathbb C}$ is a complete intersection if $C$ is defined by $n-1$ independent polynomials in $\mathbb C[z_1,\dots, z_n]$.
\end{defn}
Note that in general set-theoretic complete intersection and complete intersection are different, but here we deal only with Zariski closed sets and thus sometimes by 'complete intersection' we will mean set-theoretic complete intersection.

\begin{thm}[\cite{Ro:H}, CH-I, Corollary 1.6]
Every algebraic set in $\mathbb A_{\mathbb C}^n$ can be uniquely expressed as a union of affine algebraic varieties, no one containing another.
\end{thm}

\begin{lem}\label{lem:ag1}

Let $C$ be an affine algebraic curve in $\mathbb A_{\mathbb C}^n$ such that $V_b=C\cap V(z_n - b)$ contains finitely many points for any $b\in\mathbb C$, where $V(z_n-b)$ is the variety generated by $z_n-b \in \mathbb C[z_1,\dots, z_n]$. Then there exists a positive integer $\widetilde k$ such that $\# (V_b)\leq \widetilde k$ for any $b\in \mathbb C$, where $\# (V_b)$ is the cardinality of the set $V_b$.
\end{lem}

\begin{proof}

We have that $V_b=C\cap V(z_n - b)$. Since $V_b$ is finite, taking closure in the projective space $\mathbb P_{\mathbb C}^n$ we have
\[
V_b \subseteq \overline{C} \cap \overline{ V(z_n-b)}\,.
\]
Now $\#\; (\overline{C} \cap \overline{ V(z_n-b)})\leq degree\,(\overline C)=\widetilde k$ (say). Hence $\#\,(V_b)\leq \widetilde k$.

\end{proof}

\begin{lem}\label{lem:ag2}
Let $X$ be an affine algebraic variety in $\mathbb A_{\mathbb C}^n$ such that $X\cap V(z_n-b)$ is a finite set for every $b\in\mathbb C$. Then $X$ is an algebraic curve.
\end{lem}

\begin{proof}

The collection $\{  V(z_n-b)\,:\,b\in \mathbb C  \}$ defines a one parameter family over $\mathbb C$ and by hypothesis $X\cap V(z_n-b)$ is finite for every $b\in\mathbb C$. Let $A= \cap_{b\in\mathbb C}\,V(z_n-b)$. Then we can define a regular morphism
$
\phi:X\setminus A \rightarrow \mathbb A_{\mathbb C}^1
$
and hence a proper flat map $\phi^*:X^* \rightarrow \mathbb A_{\mathbb C}^1$, where $X^*$ is the blow up of $X$ along $A$ such that the fibre
\[
X_b^*={\phi^*}{-1}(\{b\})=X\cap V(z_n-b).
\]
Since by hypothesis $X\cap V(z_n-b)$ is a finite set, it follows that $\dim \,(X^*)=\{0\}$. Since $\phi^*$ is proper and flat, we also have
\[
\dim\,(X^*_b) =\dim\,(X^*)-\dim \,(\mathbb A_{\mathbb C}^1).
\]
So we have $\dim\,(X^*)=\dim \,(\mathbb A_{\mathbb C}^1)$. Since $X^*$ is birational with $X$, it follows that
\[
\dim\,(X)=\dim\,(X^*)=1.
\]
Hence $X$ is an algebraic curve.
\end{proof}

\begin{defn}
Let $M$ be an $R$-module. An element $a\in R$ is called an $M$-\textit{regular element} (or a \textit{non-zerodivisor} of M) if $ax=0$ with $x\in M$ implies $x=0$. A sequence $\{ a_0,\dots , a_m \}\;\; (m\geq 0)$ of elements of $R$ is called an $M$-\textit{regular sequence} if 
\begin{itemize}
\item[(a)] $M\neq (a_0,\dots, a_m)\,.M$,
\item[(b)] for $i=0,\dots ,m-1$, $a_{i+1}$ is not a zero divisor of $M/(a_0,\dots, a_i).\,M$.
\end{itemize} 

\end{defn}

\begin{thm}\label{thm:Regular}
Let $(R, \mathcal M)$ be a Noetherian local ring which is Cohen-Macaulay and suppose $\{ a_0,\dots ,a_m \} $ is a sequence of elements in $\mathcal M$. If $I=(a_0,\dots ,a_m)$, then the following statements are equivalent.
\begin{itemize}
\item[(a)] $\{ a_0,\dots, a_m \}$ is a regular sequence.
\item[(b)] $\{ a_0,\dots ,a_m \}$ is independent in $R$, that is, $I$ is a complete intersection.
\end{itemize}
\end{thm}
This is an elementary result and a proof to this result can be found in any textbook of algebraic geometry, e.g. \cite{Kunj}, Corollary 5.13 of Chapter-V.

\subsection{Representing distinguished varieties in the symmetrized polydisc}

In this subsection, we shall use the same notations and terminologies
as in \cite{AM05} introduced by Agler and M$^{\textup{c}}$Carthy.
We say that a function $f$ is \textit{holomorphic} on a
distinguished variety $\Lambda$ in $\mathbb G_n$, if for every
point in $\Lambda$, there is an open ball $B$ in $\mathbb C^n$
containing the point and a holomorphic function $F$ in $n$
variables on $B$ such that $F|_{B\cap \Lambda}=f|_{B \cap
\Lambda}$. We shall denote by $A(\Lambda)$ the Banach algebra of
functions that are holomorphic on $\Lambda$ and continuous on
$\overline{\Lambda}$. This is a closed unital subalgebra of
$C(\partial \Lambda)$ that separates points. The maximal ideal
space of $A(\Lambda)$ is $\overline{\Lambda}$.

For a finite measure $\mu$ on $\Lambda$, let $H^2(\mu)$ be the
closure of polynomials in $L^2(\partial \Lambda, \mu)$. If $G$ is
an open subset of a Riemann surface $S$ and $\nu$ is a finite
measure on $\overline G$, let $\mathcal A^2(\nu)$ denote the
closure in $L^2(\partial G, \nu)$ of $A(G)$. A point $\lambda$ is
said to be a \textit{bounded point evaluation} for $H^2(\mu)$ or
$\mathcal A^2(\nu)$ if evaluation at $\lambda$, \textit{a priori}
defined on a dense set of analytic functions, extends continuously
to the whole Hilbert space $H^2(\mu)$ or $\mathcal A^2(\nu)$
respectively. If $\lambda$ is a bounded point evaluation, then the
function defined by
$$ f(\lambda)=\langle f,k_{\lambda} \rangle $$
is called the \textit{evaluation functional at} $\lambda$. The following result
is due to Agler and M$^{\textup{c}}$Carthy.

\begin{lem}[\cite{AM05}, Lemma 1.1]\label{basiclem1}
Let $S$ be a compact Riemann surface. Let $G\subseteq S$ be a
domain whose boundary is a finite union of piecewise smooth Jordan
curves. Then there exists a measure $\nu$ on $\partial G$ such
that every point $\lambda$ in $G$ is a bounded point evaluation
for $\mathcal A^2(\nu)$ and such that the linear span of the
evaluation functional is dense in $\mathcal A^2(\nu)$.
\end{lem}

\begin{lem}\label{basiclem2}
Let $\Lambda$ be a one-dimensional distinguished algebraic set in
$\mathbb G_n$. Then there exists a measure $\mu$ on $\partial
\Lambda$ such that every point in $\Lambda$ is a bounded point
evaluation for $H^2(\mu)$ and such that the span of the bounded
evaluation functionals is dense in $H^2(\mu)$.
\end{lem}
\begin{proof}
Agler and M$^{\textup{c}}$Carthy proved a similar result for
distinguished varieties in the bidisc (see Lemma 1.2 in
\cite{AM05}); we apply similar technique here to establish the
result for the symmetrized polydisc. Let $f_1,\dots , f_{n-1}$ be
the linearly independent minimal polynomials such that
\[
\Lambda=\{(s_1,\dots,s_{n-1},p)\in \mathbb G_n\,:\,
f_i(s_1,\dots,s_{n-1},p)=0\;,\; i=1,\dots,n-1\}.
\]
Let $\mathbb Z_{f}$ be the intersection of the zero sets of
$f_1,\dots,f_{n-1}$, i.e, $\mathbb Z_{f}=\mathbb Z_{f_1} \cap
\dots \cap \mathbb Z_{f_{n-1}}$. Let $C(\mathbb Z_{f})$ be the
closure of $\mathbb Z_{f}$ in the projective space
$\mathbb{CP}^n$. Let $S$ be the desingularization of $C(\mathbb
Z_{f})$. See, e.g., \cite{fischer}, \cite{harris} and
\cite{griffiths} for details of desingularization. Therefore, $S$
is a compact Riemann surface and there is a holomorphic map $\tau:
S \rightarrow C(\mathbb Z_{f})$ that is biholomorphic from
$S^{\prime}$ onto $C(\mathbb Z_{f})^{\prime}$ and finite-to-one
from $S\setminus S^{\prime}$ onto $C(\mathbb Z_{f})\setminus
C(\mathbb Z_{f})^{\prime}$. Here $C(\mathbb Z_{f})^{\prime}$ is
the set of non-singular points in $C(\mathbb Z_{f})$ and
$S^{\prime}$ is the pre-image of $C(\mathbb Z_{f})^{\prime}$ under
$\tau$.

Let $G=\tau^{-1}(\Lambda)$. Then $\partial G$ is a finite union of
disjoint curves, each of which is analytic except possibly at a
finite number of cusps and $G$ satisfies the conditions of Lemma
\ref{basiclem1}. So there exists a measure $\nu$ on $\partial G$
such that every point in $G$ is a bounded point evaluation for
$\mathcal A^2(\nu)$. Let us define our desired measure $\mu$ by
\[
\mu(E)=\nu(\tau^{-1}(E)), \text{ for a Borel subset } E \text{ of
} \partial \Lambda.
\]
Clearly, if $\lambda$ is in $G$ and $\tau(\eta)=\lambda$, let
$k_{\eta}\nu$ be a representing measure for $\eta$ in $A(G)$. Then
the function $k_{\eta}\circ \tau^{-1}$ is defined $\mu$-almost
everywhere and satisfies
\begin{gather*}
\int_{\partial \Lambda}f(k_{\eta}\circ \tau^{-1})d\mu=
\int_{\partial G}(f\circ \tau)k_{\eta}d\nu =f\circ \tau
(\eta)=f(\lambda) \text{ and}\\
\int_{\partial \Omega}g(k_{\eta}\circ \tau^{-1})d\mu=
\int_{\partial G}(g\circ \tau)k_{\eta}d\nu =g\circ \tau
(\eta)=g(\lambda).
\end{gather*}

\end{proof}

\begin{lem}\label{lemeval}
Let $\Lambda$ be a one-dimensional distinguished algebraic set in
$\mathbb G_n$, and let $\mu$ be the measure on $\partial \Lambda$
given as in Lemma \textup{\ref{basiclem2}}. A point
$(y_1,\dots,y_n) \in \mathbb G_n$ is in $\Lambda$ if and only if
$(\bar y_1, \dots, \bar y_n)$ is a joint eigenvalue for
$M_{s_1}^*,\dots, M_{s_{n-1}}^*$ and $M_{p}^*$.
\end{lem}
\begin{proof}
It is a well known fact in the theory of reproducing kernel
Hilbert spaces that $M_f^* k_x = \overline{f(x)} k_x$ for every
multiplier $f$ and every kernel function $k_x$; in particular
every point $(\bar y_1, \dots, \bar y_n)$ in $\Lambda$ is a joint
eigenvalue for $M_{s_1}^*,\dots, M_{s_{n-1}}^*$ and $M_{p}^*$.

Conversely, if $(\bar y_1, \dots, \bar y_n)$ is a joint eigenvalue
and $v$ is a unit eigenvector, then $f(y_1,\dots,y_n) = \langle v,
M_f^* v\rangle$ for every polynomial $f$. Therefore,
\[
|f(y_1,\dots,y_n)| \leq \|M_f\| = \sup_{(s_1,\dots,s_{n-1},p) \in
\Lambda} |f(s_1,\dots,s_{n-1},p)|.
\]
So $(y_1,\dots,y_n)$ is in the polynomial convex hull of $\Lambda$
(relative to $\mathbb G_n$), which is $\Lambda$.
\end{proof}

%%%%%%%%%%%%%%%%%%%%
\begin{lem}\label{lempure}
Let $\Lambda$ be a one-dimensional distinguished algebraic set in
$\mathbb G_n$, and let $\mu$ be the measure on $\partial \Lambda$
given as in Lemma \textup{\ref{basiclem2}}. The multiplication
operator tuple $(M_{s_1},\dots, M_{s_{n-1}},M_{p})$ on $H^2(\mu)$,
defined as multiplication by the co-ordinate functions, is a pure
$\Gamma_n$-isometry.
\end{lem}
\begin{proof}
Let us consider the tuple of operators
$(\widehat{M_{s_1}},\dots,\widehat{M_{s_{n-1}}},\widehat{M_{p}})$,
multiplication by co-ordinate functions, on $L^2(\partial \Lambda,
\mu)$. They are commuting normal operators and the Taylor joint
spectrum
\[
\sigma_T(\widehat{M_{s_1}},\dots,\widehat{M_{s_{n-1}}},\widehat{M_{p}}) \subseteq \partial \Lambda \subseteq b\Gamma_n.
\]
Therefore,
$(\widehat{M_{s_1}},\dots,\widehat{M_{s_{n-1}}},\widehat{M_{p}})$
is a $\Gamma_n$-unitary and $(M_{s_1},\dots,M_{s_{n-1}},M_{p})$,
being the restriction of
$(\widehat{M_{s_1}},\dots,\widehat{M_{s_{n-1}}},\widehat{M_{p}})$
to the common invariant subspace $H^2(\mu)$, is a
$\Gamma_n$-isometry. By a standard computation, for every
$\overline y=(y_1,\dots,y_n) \in \Lambda$, the kernel function
$k_{\bar y}$ is an eigenfunction of $M_{p}^*$ corresponding to the
eigenvalue $\overline{y_n}$. Therefore,
\[
(M_{p}^*)^ik_{\overline y}=\overline{ y_n}^ik_{\overline y}
\rightarrow 0 \; \textup{ as } i \rightarrow \infty,
\]
because $|y_n|<1$ by Theorem \ref{char-G}. Since the evaluation
functionals $k_{\overline y}$ are dense in $H^2(\mu)$, this shows
that $M_{p}$ is pure. Hence $M_{p}$ is a pure isometry and
consequently $(M_{s_1},\dots, M_{s_{n-1}},M_{p})$ is a pure
$\Gamma_n$-isometry on $H^2(\mu)$.
\end{proof}

The following result will be used in the main results of this paper. Its proof is a routine exercise though we provide a short proof here for the sake of completeness.
 
\begin{lem}\label{poly-convex}

If $X\subseteq \mathbb C^n$ is a polynomially convex set, then $X$
is a spectral set for a commuting tuple $(T_1,\dots,T_n)$ if and
only if

\begin{equation}\label{pT}
\|f(T_1,\dots,T_n)\|\leq \| f \|_{\infty,\, X}\,
\end{equation}
for all holomorphic polynomials $f$ in $n$-variables.

\end{lem}

\begin{proof}

If $X$ is a spectral set for $(T_1,\dots, T_n)$, then
$\|f(T_1,\dots,T_n)\|\leq \| f \|_{\infty, X}$ follows from
definition.\\
Conversely, if the Taylor joint spectrum $\sigma_T(T_1,\dots,T_n)$
is not contained in $X$, then there is a point $(\alpha_1,\dots,
\alpha_n)$ in $\sigma_T(T_1,\dots,T_n)$ that is not in $X$. By
polynomial convexity of $X$, there is a polynomial $g$ in
$n$-variables such that $ | g(\alpha_1, \dots, \alpha_n) |
> \| g \|_{\infty, X}$. By spectral mapping
theorem,
$$ \sigma_T (g(T_1,\dots,T_n)) = \{ g(z_1,\dots, z_n ) : (z_1,\dots, z_n )
\in \sigma_T (T_1,\dots,T_n) \}$$ and hence the spectral radius of
$g(T_1,\dots, T_n)$ is bigger than $\| g \|_{\infty, X}$. But then we have that
$ \| g(T_1,\dots,T_n )\| > \| g \|_{\infty, X}$, which contradicts the
fact that $X$ is a spectral set for $(T_1,\dots,T_n)$. By polynomial convexity of $X$, a tuple satisfying (\ref{pT}) will
also satisfy
$$ \|f(T_1,\dots,T_n)\|\leq \| f \|_{\infty, X} $$
for any function holomorphic in a neighbourhood of $X$. Indeed,
Oka-Weil theorem (Theorem 5.1 in \cite{Gamelin}) allows us to
approximate $f$ uniformly by polynomials. The rest of the proof
follows by an application of Theorem 9.9 of Chapter III of
\cite{vasilescu} which deals with functional calculus in several
commuting operators.

\end{proof}

\noindent The following useful result is intuitive and comes up with a straight-forward proof.

\begin{lem}\label{lem:charming}

Let $\varphi_1, \dots,\varphi_n$ be functions in
$H^{\infty}(\mathcal B(E))$ for some Hilbert space $E$ and let
$T_{\varphi_1},\dots,T_{\varphi_n}$ be the Toeplitz
operators on $H^2(E)$. Then $(T_{\varphi_1}, \dots,
T_{\varphi_n})$ is a $\Gamma_n$-contraction if and only if
$(\varphi_1,\dots,\varphi_n)$ is a $\Gamma_n$-contraction.

\end{lem}

\begin{proof}
It is obvious that $\|T_{\varphi}\|=\| \varphi \|_{\infty}$ for
any $\varphi \in H^{\infty}(\mathcal B(E))$. For any polynomial
$p(z_1,\dots,z_n)$,
\[
p(T_{\varphi_1},\dots,T_{\varphi_n})=T_{p(\varphi_1,\dots,\varphi_n)}.
\]
Thus, by an application of Lemma \ref{poly-convex}, it suffices to have von-Neumann's inequality on $\Gamma_n$ only for the polynomials. Needless to mention that if von-Neumann's inequality holds for the tuple $(T_{\varphi_1}, \dots, T_{\varphi_n})$, then so it does for the tuple $(\varphi_1,\dots,\varphi_n)$ and vice versa.

\end{proof}

\noindent We now present one of the main results of this article, a representation of a distinguished variety in $\mathbb G_n$ in terms of the natural coordinates of $\mathbb G_n$.

\begin{thm}\label{thm:DVchar}
Let
\begin{align}\label{eq:W}
\Lambda = & \{ (s_1,\dots,s_{n-1},p)\in \mathbb G_n \,: \nonumber
\\& \;\; (s_1,\dots,s_{n-1}) \in \sigma_T(F_1^*+pF_{n-1}\,,\,
F_2^*+pF_{n-2}\,,\,\dots\,, F_{n-1}^*+pF_1) \},
\end{align}
where $F_1,\dots,F_{n-1}$ are complex square matrices of same order satisfying the following conditions:
\begin{itemize}
\item[(i)] $[F_i,F_j]=0$ and $[F_i^*,F_{n-j}]=[F_j^*,F_{n-i}]$, for $1\leq i<j\leq
n-1$ ; \item[(ii)] for every $p\in \D$ and for any unit joint-eigenvector $h$ of $(F_1^*+pF_{n-1}, F_2^*+pF_{n-2}, \dots , F_{n-1}^*+pF_1)$, the tuple $(\la F_1^*h,h \ra, \dots , \la F_{n-1}^*h,h \ra) \in \mathbb G_{n-1}$ ;
\item[(iii)] the polynomials $\{ f_1,\dots ,f_{n-1} \}$, where $f_i=\det\,(F_i^*+pF_{n-i}-s_iI)$ for each $i$, form a regular sequence ;
\item[(iv)] the complex algebraic set generated by the polynomials $S=\{ f_1,\dots, f_{n-1} \}$ is irreducible.
\end{itemize}
Then, $\Lambda$ is a distinguished variety in $\mathbb G_n$. Furthermore, $\Lambda$ is a part of an affine algebraic curve which is a set-theoretic complete intersection.

Conversely, every distinguished variety $\Lambda$ in $\mathbb G_n$ is a part of an affine algebraic curve lying in $\gn$ which is a complete intersection and has representation as in $($\ref{eq:W}$)$, where $F_1,\dots, F_{n-1}$ are complex square matrices of same order satisfying the above conditions ${(i)-(iv)}$.

\end{thm}

\begin{proof}
Suppose that
\begingroup
\begin{align*}
\Lambda  = &\{ (s_1,\dots,s_{n-1},p)\in \mathbb G_n \,: \nonumber
\\& \;\; (s_1,\dots,s_{n-1}) \in \sigma_T(F_1^*+pF_{n-1}\,,\,
F_2^*+pF_{n-2}\,,\,\dots\,, F_{n-1}^*+pF_1) \},
\end{align*}
\endgroup
where $F_1,\dots,F_{n-1}$ are complex square matrices of same
order defined on a finite dimensional Hilbert space $E$ and satisfying the given conditions. By conditions (iii) and (iv), $V_S$ is a complex algebraic variety. For any $i,j\in \{ 1,\dots,n-1 \}$, $F_i^*+pF_{n-i}$ and $F_j^*+pF_{n-j}$ commute by the given
condition-(i) for any $p\in \mathbb C$ and consequently
\[
\sigma_T(F_1^*+pF_{n-1}\,,\, F_2^*+pF_{n-2}\,,\,\dots\,,
F_{n-1}^*+pF_1)\neq \emptyset.
\]
We now show that if $|p|<1$ and if
$(s_1,\dots,s_{n-1})\in\sigma_T(F_1^*+pF_{n-1}\,,\,
F_2^*+pF_{n-2}\,,\,\dots\,, F_{n-1}^*+pF_1)$, then
$(s_1,\dots,s_{n-1},p)\in \mathbb G_n$ which will establish that
$\Lambda$ is non-empty and that it exits through the distinguished
boundary $b\Gamma_n$. This is because, proving the fact that
$\Lambda$ exits through $b\Gamma_n$ is same as proving that
$\overline{\Lambda}\cap (\partial \Gamma_n \setminus
b\Gamma_n)=\emptyset$, i.e, if
$(s_1,\dots,s_{n-1},p)\in\overline{\Lambda}$ and $|p|<1$ then
$(s_1,\dots,s_{n-1},p)\in \mathbb G_n$ (by Theorem
\ref{char-G}). Let $|p|<1$ and $(s_1,\dots,s_{n-1})$ be a joint eigenvalue of
$(F_1^*+pF_2, \dots, F_{n-1}^*+pF_1)$. Then there is a unit joint-eigenvector say $x$ such that $(F_i^*+pF_{n-i})x=s_ix$ for each $i=1, \dots , n-1$. Taking inner product with $x$ on both sides, we have for each $i$
\[
\la F_i^*x,x \ra +p \la F_{n-i}x,x \ra =s_i.
\]
Setting $\alpha_i = \la F_i^*x,x \ra$ for each $i$, we have by condition-(ii) that $(\alpha_1, \dots , \alpha_{n-1}) \in \mathbb G_{n-1}$. Thus
\[
(s_1, \dots ,s_{n-1},p)=(\alpha_1+\ov{\alpha}_{n-1}p, \dots , \alpha_{n-1}+\ov{\alpha}_{1}p , p)
\]
with $(\alpha_1, \dots , \alpha_{n-1})\in \mathbb G_{n-1}$. So, it follows from Theorem \ref{char-G} that $(s_1, \dots , s_{n-1}, p) \in \gn$. Thus,
$\Lambda$ is non-empty and it exits through the distinguished
boundary $b\Gamma_n$. Again since by condition-(iii) $\{
f_1,\dots , f_{n-1} \}$ is a regular sequence, it follows that $\Lambda$ is a complete intersection. Since $V_S$ is the affince algebraic variety generated by $f_1,\dots,f_{n-1}$, it follows that $V_S$ has complex dimension $1$. Thus, $\Lambda \;(=V_S \cap \mathbb G_n)$ is a distinguished variety in $\mathbb G_n$ which is also a part of an affine algebraic curve.\\

Conversely, let $\Lambda$ be a distinguished variety in $\mathbb
G_n$ and let $\Lambda=V_S\cap \mathbb G_n$ for a set of
polynomials
$S\subseteq \mathbb C[z_1,\dots,z_n]$. We first show that
$\Lambda$ has complex-dimension $1$. Since $V_S$ is an affine algebraic variety in
$\mathbb C^n$, by a famous theorem due to Eisenbud and Evans (see \cite{Eisenbud}), there exist $n$ polynomials $g_1, \dots , g_n \in \mathbb C[z_1,\dots , z_n]$ generating $V_S$. Therefore, without loss of
generality we can choose $S$ to be the set $\{g_1,\dots,g_n\}$ and thus have
\[
\Lambda =\{ (s_1,\dots,s_{n-1},p)\in \mathbb G_n \,:\,
g_i(s_1,\dots,s_{n-1},p)=0\,,\, i=1,\dots,n \}.
\]
Let if possible $V_S$ has complex-dimension $k$, where $1<k\leq n$. We
show that $\overline{\Lambda}$ has intersection with $\partial
\Gamma_n \setminus b\Gamma_n$, which proves that $\Lambda$ does not
exit through the distinguished boundary. Let
$(t_1,\dots,t_{n-1},q)\in\Lambda$. Therefore, $|q|<1$. Consider
the set
\[
\Lambda_{q}=\{ (s_1,\dots,s_{n-1},p)\in\Lambda\,:\,p=q \}.
\]
Clearly $\Lambda_q$ is nonempty as $(t_1,\dots,t_{n-1},q)\in\Lambda$.
Such $(s_1,\dots,s_{n-1})$ are zeros of the polynomial
$g_i(z_1,\dots,z_{n-1},q)$ for $i=1,\dots,n$. Let
\[
V_q=\{ (z_1,\dots,z_n)\in V_S \;:\; z_n=q \}.
\]
Then $\Lambda_q = V_q  \cap  \mathbb G_n$. If each such $\Lambda_q$, whenever $(t_1, \dots , t_{n-1},q)\in \Lambda$ contains only finitely many points, then by an argument similar to that in Lemma \ref{lem:ag2}, we have that $\Lambda$ has complex dimension $1$, a contradiction as $k>1$. Thus, we choose such $(t_1, \dots , t_{n-1},q)\in \Lambda$ for which $\Lambda_q$ has complex dimension at least $1$. Since $V_S$ has complex
dimension $k$, it follows that $V_q$ has dimension at least
$k-1$ which is greater than or equal to $1$. Again since each
$(s_1,\dots,s_{n-1},q)$ in $\Lambda_{q}$ is a point in $\mathbb
G_n$, by Theorem \ref{char-G}, there exists
$(\beta_1,\dots,\beta_{n-1})\in\mathbb G_{n-1}$ such that
\[
s_i=\beta_i+\bar{\beta}_{n-i}q \;,\; i=1,\dots,n-1\,.
\]
Let us consider the map
\begin{align*}
\Psi\,:\, \mathbb C^{n-1} &\rightarrow \mathbb C^{n-1} \\
(\beta_1,\dots,\beta_{n-1}) &\mapsto
(\beta_1+\bar{\beta}_{n-1}q,\beta_2+\bar{\beta}_{n-2}q,\dots,\beta_{n-1}+\bar{\beta}_{1}q).
\end{align*}
It is evident that the points $(s_1,\dots,s_{n-1})$ for which
$(s_1,\dots,s_{n-1},q)\in \Lambda_{q}$ lie inside $\Psi (\mathbb
G_{n-1})$. Also it is clear that $\Psi$ maps $\mathbb G_{n-1}$
into the following cartesian product of the scaled disks
$\binom{n}{1}\mathbb D \times \binom{n}{2} \mathbb D \times \dots
\times \binom{n}{n-1} \mathbb D$. This map $\Psi$ is real-linear
and invertible when considered a map from $\mathbb R^{2(n-1)}$ to
$\mathbb R^{2(n-1)}$, in fact a homeomorphism of $R^{2(n-1)}$.
Therefore, $\Psi$ is open and it maps the topological boundary
$\partial \Gamma_{n-1}$ of $\mathbb G_{n-1}$ onto the
topological boundary of $\Psi(\mathbb G_{n-1})$. Therefore, $V_q$
is a variety of dimension at least $1$ and a part of $V_q$ lies inside
$\Psi(\mathbb G_{n-1})$. Therefore, $V_q$ intersects the
topological boundary $\partial \overline{\Psi(\mathbb G_{n-1})}$ of the domain $\Psi(\mathbb G_{n-1})$. Since
$\Psi$ is an open map, the topological boundary of $\Psi(\mathbb
G_{n-1})$ is precisely the image of the topological boundary of
$\mathbb G_{n-1}$ under $\Psi$, that is, $\partial \overline{\Psi(\mathbb G_{n-1})}=\Psi (\partial \Gamma_{n-1})$. Take one such point say
$(\lambda_1,\dots,\lambda_{n-1})$ from the intersection of $V_q$
and $\Psi(\partial \Gamma_{n-1})$. Therefore,
\[
\lambda_i=\alpha_i+\bar{\alpha}_{n-i}q \,,\,
 \text{ for some }
(\alpha_1,\dots,\alpha_{n-1})\in \partial \Gamma_{n-1} \,(=\Gamma_{n-1} \setminus \mathbb
G_{n-1}).
\]
Since $(\alpha_1,\dots,\alpha_{n-1})\in\Gamma_{n-1} \setminus
\mathbb G_{n-1}$, by Theorem \ref{char-G},
$(\lambda_1,\dots,\lambda_{n-1},q)\in \Gamma_{n}\setminus \mathbb
G_{n}=\partial \Gamma_{n}$. Thus
$(\lambda_1,\dots,\lambda_{n-1},q)\in \overline{\Lambda}\cap
\partial\Gamma_{n}$. Again since $|q|<1$,
$(\lambda_1,\dots,\lambda_{n-1},q)$ can not lie on the
distinguished boundary $b\Gamma_n$ (by Theorem \ref{thm:DB}) and
hence $(\lambda_1,\dots,\lambda_{n-1},q)\in
\partial \Gamma_n \setminus b\Gamma_n$. Thus, $\Lambda$ does not exit
through the distinguished boundary of $\Gamma_n$ and consequently
$\Lambda$ is not a distinguished variety, a contradiction. So,
there is no distinguished variety in $\mathbb G_n$ having complex
dimension more than $1$. Thus $\Lambda$ is one-dimensional and hence a part of an affine algebraic curve lying in $\mathbb G_n$. We have that
\[
\Lambda=\{ (s_1,\dots, s_{n-1},p)\in \mathbb G_n\,:\, g_i(s_1,\dots , s_{n-1},p)=0\,, i=1,\dots , n \}.
\]
We show that the polynomials $g_1,\dots ,g_n$ do not have a non-constant common factor that intersects $\mathbb G_n$. If $g_1,\dots , g_n$ have a common non-constant factor $g$ that intersects $\mathbb G_n$, then by the argument given above, $g$ will intersect $\partial \Gamma_n \setminus b\Gamma_n$ and consequently that will contradict the fact that $\Lambda$ is a distinguished variety. So, without loss of generality we may assume that $g_1,\dots , g_n$ do not have any non-constant common factor.

Since $V_S$ has complex dimension $1$, if the $n$-th variable $p$ is held fixed, then the intersection of $g_1, \dots , g_n$ is either whole $V_S$ or consists of just finitely many points. If for some particular $p= \tilde{p}$ the intersection of $g_1(s_1,\dots ,s_{n-1}, \tilde{p}),\dots ,g_n(s_1,\dots ,s_{n-1}, \tilde{p})$ is equal to $V_S$, then $|\tilde{p}|<1$, because, $V_S$ intersects $\mathbb G_n$ and for each point $(s_1,\dots, s_{n-1},p)\in \mathbb G_n$, $p$ belongs to $\mathbb D$. This contradicts the fact that $V_S$ exits through the distinguished boundary $b\Gamma_n$ of $\Gamma_n$ as $|p|=1$ for each $(s_1,\dots , s_{n-1},p)\in b\Gamma_n$. Thus for any constant $\tilde p$, the intersection of $\{ g_i(s_1,\dots ,s_{n-1}, \tilde{p} )\,:\,i=1,\dots ,n \}$ contains finitely many, say $k$ number of points. By Lemma \ref{lem:ag1}, there exists a positive integer $\tilde k$ such that $k \leq \widetilde k$. If these $k$ points are
\begin{gather*}
\left( s_{11}(\tilde p), s_{21}(\tilde p),\dots ,s_{n-1\,1}(\tilde p),\tilde p \right), (s_{12}(\tilde p),s_{22}(\tilde p),\\
\dots, s_{n-1\,2}(\tilde p),\tilde p), \dots , (s_{1k}(\tilde p), s_{2k}(\tilde p), \dots ,s_{n-1 \, k}(\tilde p), \tilde p),
\end{gather*}
then the points
\begin{gather*}
(s_{11}(\tilde p), s_{21}(\tilde p),\dots ,s_{n-1\,1}(\tilde p)), (s_{12}(\tilde p),s_{22}(\tilde p),\dots, s_{n-1\,2}(\tilde p)),\\
\dots , (s_{1k}(\tilde p), s_{2k}(\tilde p), \dots ,s_{n-1 \, k}(\tilde p))
\end{gather*}
lies in the intersection of the sets of zeros of the polynomials
\begin{align*}
&(s_1-s_{11}(\tilde p))(s_1-s_{12}(\tilde p))\dots (s_1-s_{1\,k}(\tilde p))\,, \\
& (s_2-s_{21}(\tilde p))(s_2-s_{22}(\tilde p))\dots (s_2-s_{2\,k}(\tilde p))\,,\\
& \vdots \\
& (s_{n-1}-s_{n-1\,1}(\tilde p))(s_{n-1}-s_{n-1\, 2}(\tilde p))\dots (s_{n-1}-s_{n-1\,k}(\tilde p))\,.
\end{align*}
Considering the first polynomial we see that 
\[
(s_1-s_{11}(\tilde p))(s_1-s_{12}(\tilde p)\dots (s_1-s_{1\,k}(\tilde p))=0
\]
and it holds for any $\tilde p \in \mathbb D$. This implies that
\[
s_1^k \in {\text{span}}\,\{ 1,s_1,\dots, s_1^{k-1} \}+{Ran}\,M_p.
\]
Since $k \leq \widetilde k$, it follows that
\[
s_1^{\widetilde k} \in {\text{span}}\,\{ 1,s_1,\dots, s_1^{{\widetilde k}-1} \}+{Ran}\,M_p.
\]
Same argument holds for $s_2,\dots, s_{n-1}$. Therefore,
\begin{equation}\label{eqn:d01}
s_i^{\widetilde k}\in {\text{span}}\,\{ 1,s_i,\dots, s_i^{{\widetilde k}-1} \}+{Ran}\,M_p\,,\; \text{ for all } i=1,\dots, n-1.
\end{equation}
Thus it follows from (\ref{eqn:d01}) that for any positive integers $k_1,\dots , k_{n-1} $,
\[
s_1^{k_1}\dots s_{n-1}^{k_{n-1}}\in \text{span}\, \{ s_1^{i_1}s_2^{i_2}\dots s_{n-1}^{i_{n-1}}\,:\,0 \leq i_1,\dots,i_{n-1} \leq {\widetilde k}-1 \} +{Ran}M_p.
\]
Hence
\[
H^2(\mu) = {\text{span}}\, \{ s_1^{i_1}s_2^{i_2}\dots s_{n-1}^{i_{n-1}}\,:\,0 \leq i_1,\dots,i_{n-1} \leq {\widetilde k}-1 \} +{Ran}M_p\,,
\]
where $\mu$ is the measure as in Lemma \ref{basiclem2}. Now $M_{p}M_{p}^*$ is a
projection onto $Ran\, M_{p}$ and
\begin{equation}\label{eqn:essn}
Ran\,M_{p} \supseteq \{ pf(s_1,\dots,s_{n-1},p):\; f \text{ is a
polynomial in } s_1,\dots,s_{n-1},p \}.
\end{equation}
Therefore, $Ran\,(I-M_{p}M_{p}^*)$, which is equal to $\mathcal D_{M_{P}^*}$, has finite dimension, say $d$. Suppose $(M_{s_1},\dots,M_{s_{n-1}},M_{p})$ on $H^2(\mu)$ is the tuple of multiplication operators with the multipliers being the
co-ordinate functions. Then by Lemma \ref{lempure},
$(M_{s_1},\dots,M_{s_{n-1}},M_{p})$ is a pure $\Gamma_n$-isometry
on $H^2(\mu)$. Since $Ran\,(I-M_{p}M_{p}^*)$ has
finite dimension
$d$, it follows from Theorem \ref{model1} that $(M_{s_1},\dots,M_{s_{n-1}},M_{p})$
can be identified with
$(T_{\varphi_1},\dots,T_{\varphi_{n-1}},T_{z})$ on $H^2(\mathcal
D_{M_{P}^*})$, where $\varphi_i(z)=F_i^*+F_{n-i}z$ for each $i$,
$(F_1,\dots,F_{n-1})$ being the fundamental operator tuple of the adjoint
$(T_{F_1^*+F_{n-1}z}^*,\dots ,T_{F_{n-1}^*+F_1z}^*,T_{z}^*)$.
 By Lemma \ref{lemeval}, a point $(t_1,\dots ,t_{n-1},q)$ is in $\Lambda$ if
 and only if $(\bar t_1, \dots ,\bar t_{n-1}, \bar q)$ is a joint
 eigenvalue of $T_{\varphi_1}^*,\dots,T_{\varphi_{n-1}}^*$ and $T_{z}^*$.
This can happen if and only if $(\bar t_1, \dots ,\bar t_{n-1})$
is a joint eigenvalue of $(\varphi_1(q)^*\,, \dots ,
\varphi_{n-1}(q)^*)$. Therefore,
\[
\Lambda= \{ (s_1,\dots,s_{n-1},p)\in \mathbb G_n \,:  (s_1,\dots,s_{n-1}) \in \sigma_T(F_1^*+pF_{n-1}\,,\,
F_2^*+pF_{n-2}\,,\,\dots\,, F_{n-1}^*+pF_1) \}.
\]
If $S^{\prime}=\{ f_1 ,\dots, f_{n-1} \}$, where $f_i=\det\,(F_i^*+pF_{n-i}-s_iI)$, then $\Lambda \subseteq V_{S^{\prime}}\cap \mathbb G_n$. Since $\Lambda=V_S\cap \mathbb G_n$, this implies that $V_S\subseteq V_{S'}$. Note that for every fixed $p\in \mathbb C$, the intersection of the family $\{ f_1,\dots,f_{n-1} \}$ contains only finitely many points. So by Lemma \ref{lem:ag2}, $V_{S'}$ is an affine algebraic curve and thus has complex dimension $1$. It follows from here that $\{ f_1,\dots,f_{n-1} \}$ forms a regular sequence. Thus by Theorem \ref{thm:Regular}, $V_{S'}$ is a complete intersection. Since $\Lambda \subseteq V_{S'}\cap \mathbb G_n$, it follows that either $\Lambda=V_{S'}\cap \mathbb G_n$ or $\Lambda$ is an irreducible component of $V_{S'}\cap \mathbb G_n$. In any case $\Lambda$ is a complete intersection. Again, since $(T_{\varphi_1},\dots ,T_{\varphi_{n-1}},T_{z})$ on
$H^2(\mathcal D_{M_{P}^*})$ is a pure $\Gamma_n$-isometry, by the
commutativity of $\varphi_1 \,,\, \dots , \varphi_{n-1}$ we have
that $[F_i , F_j]=0$ for each $i,j$ and that
$[F_i^*,F_{n-j}]=[F_j^*,F_{n-i}]$ for $1 \leq i<j \leq n-1$. Also, for each $z \in \T$, $(\varphi_1(z), \dots , \varphi_{n-1}(z), zI)$ is a $\Gamma_n$-contraction, by Lemma \ref{lem:charming}. Consider the holomorphic map $g(z)= (\varphi_1(z), \dots , \varphi_{n-1}(z), zI)$ with $z \in \ov{\D}$. For any holomorphic polynomial $f(s_1, \dots ,s_{n-1},p)$ we have that
\[
\|f(\varphi_1(z), \dots , \varphi_{n-1}(z), zI) \| \leq \|f\|_{\infty, \Gamma_n} \text{ for any } z\in \T
\] 
which further implies by Maximum principle implies that
\[
\|f(\varphi_1(z), \dots , \varphi_{n-1}(z), zI) \| \leq \|f\|_{\infty, \Gamma_n} \text{ for any } z\in \D.
\]
Therefore, $(\varphi_1(z), \dots , \varphi_{n-1}(z), zI)$ is a $\Gamma_n$-contraction for every $z\in \D$ by Lemma \ref{poly-convex} and consequently for every $z\in \D$ we have that
\[
\sigma_T (\varphi_1(z), \dots , \varphi_{n-1}(z), zI)= \sigma_T\,(F_1^*+zF_{n-1}\,,\,F_2^*+zF_{n-2}\,,\,\dots\,, F_{n-1}^*+zF_1\,,\,zI) \subseteq \Gamma_n.
\]
For $|p|<1$, if
\[
(s_1,\dots , s_{n-1},p)\in \sigma_T\,(F_1^*+pF_{n-1}\,,\,F_2^*+pF_{n-2}\,,\,\dots\,, F_{n-1}^*+pF_1,pI)\cap \partial \Gamma_n\,,
\]
then $(s_1,\dots ,s_{n-1},p)\in V_S \cap (\partial \Gamma_n \setminus b\Gamma_n )$ as $|p|<1$. This implies that $\overline{\Lambda}\cap (\partial \Gamma_n \setminus b\Gamma_n ) \neq \emptyset$, a contradiction to the fact that $\Lambda$ exits through the distinguished boundary $b\Gamma_n$. Therefore,
\[
\sigma_T\,(F_1^*+zF_{n-1}\,,\,F_2^*+zF_{n-2}\,,\,\dots\,, F_{n-1}^*+zF_1\,,\,zI) \subseteq \mathbb G_n.
\]
 For $p\in \D$, if $h$ is a unit joint eigenvector of $(F_1^*+pF_{n-1}\,,\,F_2^*+pF_{n-2}\,,\,\dots\,, F_{n-1}^*+pF_1)$ corresponding to the joint eigenvalue $(s_1, \dots, s_{n-1})$, then evidently $h$ is a joint eigenvector of $(F_1^*+pF_{n-1}\,,\,F_2^*+pF_{n-2}\,,\,\dots\,, F_{n-1}^*+pF_1, pI)$ corresponding to the joint eigenvalue $(s_1, \dots, s_{n-1},p)$. Clearly $(s_1, \dots , s_{n-1},p) \in \mathbb G_n$ by the argument given above. Now we have that $(F_i^*+pF_{n-i})h=s_ih$, which gives
\[
\la F_i^*h, h \ra + p \la F_{n-i}h, h \ra = s_i.
\]
Thus, if $\alpha_i= \la F_i^* h, h \ra$, then $s_i=\alpha_i + \ov{\alpha}_{n-i}p$ for each $i=1 , \dots , n-1$. Again since $(s_1, \dots , s_{n-1},p)$ is in $\gn$, it follows that $(\alpha_1, \dots ,\alpha_{n-1})= (\la F_1^*h,h \ra , \la F_2^*h, h \ra, \dots , \la F_{n-1}^*h,h \ra) \in \mathbb G_{n-1}$. The proof is now complete.

\end{proof}

\vspace{0.2cm}

\noindent \textbf{\textit{An alternative characterization for the distinguished varieties.}} Here we have the same representation of distinguished varieties in $\gn$ but under a different necessary and sufficient condition. Indeed, one of the conditions in Theorem \ref{thm:DVchar} is replaced by a new one here.

\begin{thm}\label{thm:DVchar-1}

Let
\begin{align}\label{eq:W1}
\Lambda = & \{ (s_1,\dots,s_{n-1},p)\in \mathbb G_n \,: \nonumber
\\& \;\; (s_1,\dots,s_{n-1}) \in \sigma_T(F_1^*+pF_{n-1}\,,\,
F_2^*+pF_{n-2}\,,\,\dots\,, F_{n-1}^*+pF_1) \},
\end{align}
where $F_1,\dots,F_{n-1}$ are complex square matrices of same order satisfying the following conditions:
\begin{itemize}
\item[(i)] $[F_i,F_j]=0$ and $[F_i^*,F_{n-j}]=[F_j^*,F_{n-i}]$, for $1\leq i<j\leq
n-1$ ; \item[(ii)] for every $z\in \D$ the Taylot joint spectrum of $(F_1^*+zF_{n-1}, F_2^*+zF_{n-2}, \dots , F_{n-1}^*+zF_1, zI)$ is contained in $\gn$ ;
\item[(iii)] the polynomials $\{ f_1,\dots ,f_{n-1} \}$, where $f_i=\det\,(F_i^*+pF_{n-i}-s_iI)$ for each $i$, form a regular sequence ;
\item[(iv)] the complex algebraic set generated by the polynomials $S=\{ f_1,\dots, f_{n-1} \}$ is irreducible.
\end{itemize}
Then, $\Lambda$ is a distinguished variety in $\mathbb G_n$. Furthermore, $\Lambda$ is a part of an affine algebraic curve which is a set-theoretic complete intersection.

Conversely, every distinguished variety $\Lambda$ in $\mathbb G_n$ is a part of an affine algebraic curve lying in $\gn$ which is a complete intersection and has representation as in $($\ref{eq:W}$)$, where $F_1,\dots, F_{n-1}$ are complex square matrices of same order satisfying the above conditions ${(i)-(iv)}$.

\end{thm}

\begin{proof}

Note that only condition-(ii) of Theorem \ref{thm:DVchar} is replaced by a new condition here. Thus, it is enough to prove that condition-(ii) of Theorem \ref{thm:DVchar} is necessary and sufficient for condition-(ii) of the present theorem. Suppose condition-(ii) of Theorem \ref{thm:DVchar} holds. Let $p\in \D$ and let $(s_1, \dots ,s_{n-1},p)$ be a joint-eigenvalue of $(F_1^*+pF_{n-1}, F_2^*+pF_{n-2}, \dots , F_{n-1}^*+pF_1, pI)$. To show that $(s_1, \dots , s_{n-1},p) \in \gn$. Clearly $(s_1, \dots ,s_{n-1})$ is a joint-eigenvalue of $(F_1^*+pF_{n-1}, F_2^*+pF_{n-2}, \dots , F_{n-1}^*+pF_1)$ and if $x$ is a unit join-eigenvector, then we have
\[
(F_i^*+pF_{n-i})x=s_ix \quad \text{ for each } i=1, \dots , n-1.
\]
Taking inner product with $x$ we have
$
\alpha_i+\ov{\alpha}_{n-i}p=s_i$, where $\alpha_i=\la F_i^*x,x \ra $ for each $i$. By condition-(ii) of Theorem \ref{thm:DVchar}, we have that $(\alpha_1, \dots ,\alpha_{n-1}) \in \mathbb G_{n-1}$ and consequently by Theorem \ref{char-G}, $(s_1, \dots ,s_{n-1},p)\in \gn$. The proof of the converse part follows from the last part of the proof of Theorem \ref{thm:DVchar}.

\end{proof}

\noindent \textbf{\textit{A refinement of the previous result for} $\mathbb G_2$:} We now show that the condition-(ii) in both Theorem \ref{thm:DVchar} and Theorem \ref{thm:DVchar-1} provides a refinement of the previous result on distinguished varieties in $\mathbb G_2$. In \cite{pal-shalit}, the author and Shalit had the following representation for the distinguished varieties in the symmetrized bidisc:

\begin{thm}[Theorem 3.5, \cite{pal-shalit}]\label{thm:DVsym}
Let $A$ be a square matrix A with numerical radius $\omega(A)< 1$,
and let $W$ be the subset of $\mathbb G_2$ defined by
\begin{equation*}
W = \{(s,p) \in \mathbb G_2 \,:\, \det(A + pA^* - sI) = 0\}.
\end{equation*}
Then $W$ is a distinguished variety. Conversely, every
distinguished variety in $\mathbb G_2$ has the form $\{(s,p) \in
\mathbb G_2 \,:\, \det(A + pA^* - sI) = 0\}$, for some matrix $A$
with $\omega(A)\leq 1$.
\end{thm}
First we consider condition-(ii) of Theorem \ref{thm:DVchar-1}. Note that $W$ may or may not be a distinguished variety in
$\mathbb G_2$ if $\omega(A)=1$ as was shown in \cite{pal-shalit}
by examples. In the present paper, we have shown that $W$ is a
distinguished variety in $\mathbb G_2$ if and only if
\begin{equation*}
W = \{(s,p) \in \mathbb G_2 \,:\, \det(A^* + pA - sI) = 0\}\,,
\end{equation*}
for a square matrix $A$ satisfying $\sigma_T(A^*+zA,zI)\subseteq \mathbb G_2$ for all $z\in\mathbb D$. The hypothesis $\omega(A)<1$ is stronger than $\sigma_T(A^*+zA,zI)\subseteq \mathbb G_2$, because, if $A$ is of order $n\times n$ and $\omega(A)\leq 1$, then by Theorem 2.12 in \cite{pal-shalit}, $A$ is the fundamental operator of the $\Gamma_2$-co-isometry $(T_{A^*+Az}^*,T_z^*)$ (see the proof of Theorem 2.12 in \cite{pal-shalit}), where the Toeplitz operators $T_{A^*+Az}^*,T_z^*$ act on $H^2(\mathbb C^n)$. Since $(T_{A^*+Az},T_z)$ is a $\Gamma_2$-contraction, it follows that $\sigma_T(A^*+Az,zI)\subseteq \Gamma_2$. In particular when $\omega(A)<1$, following the proof of Theorem \ref{thm:DVsym} from \cite{pal-shalit} we see that $\sigma_T(A^*+Az,zI)\in\mathbb G_2$ for all $z\in \mathbb D$. Thus, the fact that $\omega(A)<1$ implies that $\sigma_T(A^*+Az,zI)\subseteq \mathbb G_2$ for all $z\in\mathbb D$.

Also, if we consider condition-(ii) of Theorem \ref{thm:DVchar}, we see that if $n=2$, it suffices to have $(\la Ax,x \ra) \in \D$ when $x$ is a unit eigenvector of $A$. This is much weaker than demanding $\omega(A)\leq 1$ i.e. $W(A)\subseteq \ov{\D}$.\\

We conclude this section by proving that the closure of every distinguished variety is polynomially convex.

\begin{prop}\label{prop:poly-convex}
The closure of any distinguished variety in $\mathbb G_n$ is polynomially convex.
\end{prop}

\begin{proof}
Suppose
\begin{align*}
 \Lambda= &\{ (s_1,\dots,s_{n-1},p)\in \mathbb G_n \,: \nonumber
\\& \quad (s_1,\dots,s_{n-1}) \in \sigma_T(F_1^*+pF_{n-1}\,,\,
F_2^*+pF_{n-2}\,,\,\dots\,, F_{n-1}^*+pF_1) \},
\end{align*}
is a distinguished variety in $\mathbb G_n$. Then for any
$(s_1,\dots,s_{n-1},p)\in \Lambda $, we have
\[
\det (F_i^*+pF_{n-i}-s_iI)=0\;,\; i=1,\dots,n-1.
\]
We show that $\overline{\Lambda}$ is polynomially convex.
Evidently $\overline{\Lambda}=\Lambda \cup
\partial \Lambda \subseteq \Gamma_n$, where $\partial \Lambda =
Z(f_1,\dots,f_{n-1})$ $\cap b\Gamma_n$ when $f_i=\det
(F_i^*+pF_{n-i}-s_iI)$ (see Theorem \ref{thm:DVchar}). Let
$z=(z_1,\dots,z_n)\in\mathbb C^n \setminus \Gamma_n$. Since
$\Gamma_n$ is polynomially convex, there is a polynomial say $f$
such that
\[
|f(z)|> \sup_{y\in\Gamma_n} |f(y)|=\| f \|_{\infty , \Gamma_n}.
\]
Therefore, $|f(z)|> \|f\|_{\infty , \overline{\Lambda}}$. Now let
$x=(\hat s_1,\dots,\hat s_{n-1},\hat p)$ be a point in $\Gamma_n \setminus \overline{\Lambda}$. Then $\det (F_i^*+ \hat pF_{n-i}-\hat s_iI) \neq 0$ for some $i$ in $\{ 1,\dots, n-1 \}$. Let $\det(F_k^*+ \hat p F_{n-k}- \hat s_kI)=\alpha \neq 0$. We choose $g=f_k= \det(F_k^*+pF_{n-k}-s_kI)$ and see that
\[
|g(x)|=|\alpha|>0 = \|  g\|_{\infty, \overline{\Lambda}}.
\]
So, we conclude that $\overline{\Lambda}$ is polynomially convex.

\end{proof}

%\begin{rem}
%The conclusion of Proposition \ref{prop:poly-convex} is true for any distinguished variety in a bounded domain. This follows easily from the proof of Proposition \ref{prop:poly-convex}.
%\end{rem}

\vspace{0.2cm}

\section{Interplay between the distinguished varieties in $\mathbb G_2$ and $\mathbb G_3$} \label{sec:06}

\vspace{0.3cm}

\noindent In this section, we show by explicit construction how a distinguished variety in $\mathbb G_3$ gives rise to a distinguished variety in $\mathbb G_2$. 

\begin{lem}\label{sc:1}
If $(s_1,s_2,p)\in\mathbb C$ is in $\Gamma_3 \;($or, respectively
in $\mathbb G_3 )$ then $(\dfrac{s_1}{3}+\omega \dfrac{s_2}{3},
\omega p)\in \Gamma_2 \;($or, respectively in $\mathbb G_2 )$ for
every $\omega\in\mathbb T$.
\end{lem}

\begin{proof}
We prove for $\Gamma_3$ and $\Gamma_2$ because the proof for
$\mathbb G_3$ and $\mathbb G_2$ is similar. Let
$(s_1,s_2,p)\in\Gamma_3$. Let $s_{\omega}=\dfrac{s_1}{3}+\omega
\dfrac{s_2}{3}$ and $p_{\omega}=\omega p$. We prove that
$(s_{\omega},p_{\omega})\in\Gamma_2$ by using Theorem 1.1 in
\cite{ay-jot} which states that a point $(s,p)\in\mathbb C^2$ is
in $\Gamma_2$ if and only if
\begin{equation}\label{eq:1}
|s|\leq 2 \text{ and } |s-\bar s p|\leq 1-|p|^2\,.
\end{equation}
Now $|s_{\omega}|\leq 2$ is obvious because we have $|s_i|\leq 3$
for $i=1,2$. Also by Theorem \ref{char-G}, there exists
$(c_1,c_2)\in\Gamma_2$ such that
\[
s_1=c_1+\bar{c_2}p \text{ and } s_2=c_2+\bar{c_1}p\,.
\]
It is evident that
\[
c_1=\dfrac{s_1-\bar{s_2}p}{1-|p|^2}\,,\, c_2 =
\dfrac{s_2-\bar{s_1}p}{1-|p|^2} \text{ and that } |c_1|+|c_2|\leq
3.
\]

Now
\begingroup
%\allowdisplaybreaks
\begin{align*}
|s_{\omega}-\bar{s_{\omega}}p_{\omega}|=\dfrac{1}{3}|(s_1+\omega
s_2)-(\bar{s_1}+\bar{\omega}\bar{s_2})\omega p)| &\leq
\dfrac{1}{3}(|s_1-\bar{s_2}p|+|s_2-\bar{s_1}p|)\\
& = \dfrac{1}{3}(|c_1|+|c_2|)(1-|p|^2)\\
& = \dfrac{1}{3}(|c_1|+|c_2|)(1-|p_{\omega}|^2) \\
& \leq (1-|p_{\omega}|^2).
\end{align*}
\endgroup
Thus $(s_{\omega},p_{\omega})\in\Gamma_2$.

\end{proof}

The converse of the above result does not hold. Let us choose
$s_1,s_2,p$ in the following way:
\[
p=\dfrac{1}{2}\;,\; s_1=1+ 2\times \dfrac{1}{2}=2\;,\; s_2=
2+\dfrac{1}{2}=\dfrac{5}{2}.
\]
It is evident here that $(c_1,c_2)=(1,2)\notin \Gamma_2$ and by
Theorem \ref{char-G}, $(s_1,s_2,p)\notin \Gamma_3$. If we follow
the same technique as in the proof of the previous lemma, it is
easy to see that $\left( \dfrac{s_1}{3}+\omega
\dfrac{s_2}{3},\omega p \right)$ is in $\Gamma_2$ for every
$\omega\in\mathbb T$.\\

An application of the previous lemma immediately gives the
following operator version of the same result.

\begin{thm}\label{op:1}
If $(S_1,S_2,P)$ is a $\Gamma_3$-contraction then
$\left(\dfrac{S_1}{3}+\omega \dfrac{S_2}{3}, \omega P \right)$ is
a $\Gamma_2$-contraction for every $\omega\in\mathbb T$.
 \end{thm}

\begin{proof}

Let $g_{\omega}$ be the map from $\Gamma_3$ to $\Gamma_2$ that
maps $(s_1,s_2,p)$ to $(s_1+\omega s_2, \omega p)$. Let
$(S_{\omega},P_{\omega})=\left(\dfrac{S_1}{3}+\omega
\dfrac{S_2}{3}, \omega P \right)$. Then for any holomorphic
polynomial $f$ in two variables we have

\[
\| f(S_{\omega},P_{\omega}) \|=\| f\circ g_{\omega}(S_1,S_2,P) \|
\leq \| f\circ g_{\omega} \|_{\infty, \Gamma_3} = \| f\|_{\infty,
g_{\omega}(\Gamma_3)} \leq \| f\|_{\infty, \Gamma_2}.
\]
Thus by Lemma \ref{poly-convex}, $(S_{\omega},P_{\omega})$ is a
$\Gamma_2$-contraction.
\end{proof}

\noindent We now present the main result of this section.

\begin{thm}\label{connection}
Let
\[
\Lambda=\{ (s_1,s_2,p)\in\mathbb G_3
\,:\,(s_1,s_2)\in\sigma_T(F_1^*+F_2p,F_2^*+F_1p) \}
\]
be a distinguished variety in $\mathbb G_3$. Then $W=\phi
(\Lambda)$ is a distinguished variety in $\mathbb G_2$, where
$\phi$ is the following holomorphic map:
\begin{align*} & \phi\,:\,
\mathbb G_3 \longrightarrow \mathbb G_2 \\& (s_1,s_2,p) \mapsto
\left( \dfrac{s_1}{3}+\dfrac{s_2}{3},p \right).
\end{align*}

\end{thm}

\begin{proof}
We have that
\[
W=\{ \left( \dfrac{s_1+s_2}{3},p \right) \,:\,(s_1,s_2,p)\in\Lambda \}.
\]
Since $(s_1,s_2)\in\sigma_T(F_1^*+F_2p,F_2^*+F_1p)$,
$\dfrac{s_1+s_2}{3}$ is an eigenvalue of
$\dfrac{(F_1+F_2)^*}{3}+p\dfrac{(F_1+F_2)}{3}$. Therefore,
\begin{align*}
W &=\left\{ (\frac{s_1+s_2}{3},p)\in\mathbb G_2: \det
\left[\frac{(F_1+F_2)^*}{3}+p\frac{(F_1+F_2)}{3}-\frac{(s_1+s_2)}{3}I \right]  =0 \right\} \\ &=\{ (s,p)\in\mathbb G_2\,:\, \det
\left[\frac{(F_1+F_2)^*}{3}+p\frac{(F_1+F_2)}{3}-sI \right]=0 \}.
\end{align*}
Now for proving $W$ a distinguished variety in $\mathbb G_2$,
it suffices to show, by Theorem \ref{thm:DVsym}, that $\omega \left
(\dfrac{F_1+F_2}{3} \right)<1$. Since
$\sigma_T(F_1^*,F_2^*)\subseteq \mathbb G_2$, for
$(s,p)\in\sigma_T(F_1^*,F_2^*)$ there is a unit joint eigenvector
$v$ such that
$
F_1^*v=sv\,,\,F_2^*v=pv.
$
Taking inner product with $v$ we get
$
\bar s=\langle F_1v,v  \rangle \;,\; \bar p=\langle  F_2v,v
\rangle.
$
Since $(\bar s, \bar p)\in\mathbb G_2$, we have that $|\bar
s|+|\bar p|<3$. Therefore,
\begin{align*}
\omega(F_1+F_2) & =\sup_{\|v\|=1} \{ |\langle (F_1+F_2)v,v
\rangle| \} \\
& \leq \sup_{\|v\|=1}\{ |\langle F_1v,v \rangle|+|\langle F_2v,v
\rangle| \} \\
& <3.
\end{align*}
This implies that $\omega \left(\dfrac{F_1+F_2}{3} \right)<1$.
Therefore, $W$ is a distinguished variety in $\mathbb G_2$.

\end{proof}

\vspace{0.2cm}

%%%%%%%%%%%%%%%%%%%%%%%%%%%%%%

\section{Model for pure $\Gamma_n$-isometries and rational dilation on the distinguished varieties}

\vspace{0.5cm}

\noindent Operator theory on any domain is always of independent interests. We recall from the literature a few important and useful results on operator theory of the symmetrized polydisc. These results will be frequently used throughout this section.

\begin{thm}[\cite{sourav6}, Theorem 3.3]\label{existence-uniqueness} Let
$(S_1,\dots,S_{n-1},P)$ be a $\Gamma_n$-contraction on a Hilbert
space $\mathcal H$. Then there are unique operators
$A_1,\dots,A_{n-1}\in\mathcal B(\mathcal D_P)$ such that
\[
S_i-S_{n-i}^*P=D_PA_iD_P \text{ for } i=1,\dots,n-1.
\]
Moreover, for each $i$ and for all $z\in \mathbb T$, $\omega
(A_i+A_{n-i}z)\leq \tilde n_i$. $[\tilde n_i= \binom{n}{i}]$.
\end{thm}
This unique $(n-1)-$tuple $(A_1,\dots, A_{n-1})$ plays central role in all sections of operator theory on the symmetrized polydisc, e.g. \cite{tirtha-sourav, sourav3, sourav6} etc. and hence is called the fundamental operator tuple or shortly the $\ft$-{tuple} of the $\Gamma_n$-contraction $(S_1,\dots,S_{n-1},P)$. The following lemma from \cite{BSR} will be used in the proof of the main result of this section.
\begin{lem}[Lemma 2.7, \cite{BSR}] \label{lem:BSR2}
If $(S_1,\dots,S_{n-1},P)$ is a $\Gamma_n$-contraction then
\[
\left(\dfrac{n-1}{n}S_1,\dfrac{n-2}{n}S_2, \dots, \frac{1}{n}S_{n-1} \right)
\]
is a $\Gamma_{n-1}$-contraction.
\end{lem}

Since the closed symmetrized polydisc $\Gamma_n$ and the closures of the distinguished varieties in $\gn$ are polynomially convex, by virtue of Lemma \ref{poly-convex}, a verification of von-Neumann's inequality on $\Gamma_n$ or any distinguished variety $\ov{\Lambda}$ suffices for a commuting operator tuple $(T_1, \dots, T_n)$ to have $\Gamma_n$ or $\ov{\Lambda}$ as a spectral set. In this Section, we shall show that $\ov{\Lambda}$ can be not only a spectral set for a $\Gamma_n$-contraction $(S_1, \dots ,S_{n-1},P)$ but also a complete spectral set for $(S_1, \dots ,S_{n-1},P)$ if $\Lambda$ is defined by  the $\ft$-tuple of $(S_1, \dots ,S_{n-1},P)$. Let us define spectral and complete spectral set below. 

\subsection{Spectral set and complete spectral set}

We shall follow Arveson's terminologies here. Let $X$ be a compact subset of $\mathbb
C^n$ and let $\mathcal R(X)$ denote the algebra of all rational
functions on $X$, that is, all quotients $p/q$ of polynomials
$p,q \in \C[z_1, \dots ,z_n]$ such that $q$ does not have any zero in $X$. The norm of an element
$f$ in $\mathcal R(X)$ is defined as
$$\|f\|_{\infty, X}=\sup \{|f(\xi)|\;:\; \xi \in X  \}. $$
Also for each $k\geq 1$, let $\mathcal R_k(X)$ denote the algebra
of all $k \times k$ matrices over $\mathcal R(X)$. Obviously each
element in $\mathcal R_k(X)$ is a $k\times k$ matrix of rational
functions $F=(f_{i,j})$ and we can define a norm on $\mathcal
R_k(X)$ in the canonical way
$$ \|F\|=\sup \{ \|F(\xi)\|\;:\; \xi\in X \}, $$ thereby making
$\mathcal R_k(X)$ into a non-commutative normed algebra. Let
$\underline{T}=(T_1,\cdots,T_n)$ be an $n$-tuple of commuting
operators on a Hilbert space $\mathcal H$. The set $X$ is said to
be a \textit{spectral set} for $\underline T$ if the Taylor joint
spectrum $\sigma_T (\underline T)$ of $\underline T$ is a subset
of $X$ and
\begin{equation}\label{defn1}
\|f(\underline T)\|\leq \|f\|_{\infty, X}\,, \textup{ for every }
f\in \mathcal R(X).
\end{equation}
Here $f(\underline T)$ can be interpreted as $p(\underline
T)q(\underline T)^{-1}$ when $f=p/q$. Moreover, $X$ is said to be
a \textit{complete spectral set} if $\|F(\underline T)\|\leq \|F\|$ for every $F$ in $\mathcal R_k(X)$, $k=1,2,\cdots$.

\subsection{Rational dilation}
A commuting $n$-tuple of operators $\underline T$ on a Hilbert
space $\mathcal H$, having $X$ as a spectral set, is said to have
a \textit{rational dilation} or \textit{normal}
$bX$-\textit{dilation} if there exists a Hilbert space $\mathcal
K$, an isometry $V:\mathcal H \rightarrow \mathcal K$ and an
$n$-tuple of commuting normal operators $\underline
N=(N_1,\cdots,N_n)$ on $\mathcal K$ with $\sigma_T(\underline
N)\subseteq bX$ such that

\begin{equation}\label{rational-dilation}
f(\underline T)=V^*f(\underline N)V, \textup{ for every } f\in
\mathcal R(X),
\end{equation}

or, in other words $f(\underline T)=P_{\mathcal H}f(\underline N)|_{\mathcal H}$ for every $f\in \mathcal R(X)$, when $\mathcal H$ is considered as a closed linear subspace of $\mathcal K$. Moreover, the dilation is called {\em minimal} if
\[
\mathcal K=\overline{\textup{span}}\{ f(\underline N) h\,:\;
h\in\mathcal H \textup{ and } f\in \mathcal R(K) \}.
\]

\begin{thm}[Arveson, \cite{Arveson-II}] \label{Arveson}
Let $X \subset \C^n$ be compact. A commuting $n$-tuple of Hilsert space operators $(T_1, \dots , T_n)$ admits a normal $\partial X$-dilation if and only if $X$ is a complete spectral set for $(T_1, \dots ,T_n)$.
\end{thm}

\begin{defn}
Let $(S_1,\dots,S_{n-1},P)$ be a $\Gamma_n$-contraction on
$\mathcal H$. A commuting $n-$tuple of operators $(Q_1,\dots,Q_{n-1},V)$ acting on
$\mathcal K$ is said to be a $\Gamma_n$-isometric dilation of
$(S_1,\dots,S_{n-1},P)$ if $\mathcal H \subseteq \mathcal K$,
$(Q_1,\dots,Q_{n-1},V)$ is a $\Gamma_n$-isometry and
\[
P_{\mathcal H}(Q_1^{m_1}\dots Q_{n-1}^{m_{n-1}}V^n)|_{\mathcal
H}=S_1^{m_1}\dots S_{n-1}^{m_{n-1}}P^n, \; m_1,\dots, m_{n-1},n
\in\mathbb N \cup \{0\}.
\]
Here $P_{\mathcal H}:\mathcal K \rightarrow \mathcal H$ is the
orthogonal projection of $\mathcal K$ onto $\mathcal H$. Moreover, since $\Gamma_n$ is polynomially convex, the dilation is minimal if
\[
\mathcal K=\overline{\textup{span}}\{ Q_1^{m_1}\dots,
Q_{n-1}^{m_{n-1}}V^n h\,:\; h\in\mathcal H \textup{ and
}m_1,\dots,m_{n-1},n\in \mathbb N \cup \{0\} \}.
\]
\end{defn}
In a similar fashion one can define a $\Gamma_n$-unitary dilation of a $\Gamma_n$-contraction. Note that if $(T_1, \dots , T_n)$ is a $\Gamma_n$-co-isometric extension of a $\Gamma_n$-contraction $(S_1, \dots ,S_{n-1},P)$, then naturally $(T_1^*, \dots ,T_n^*)$ is a $\Gamma_n$-isometric dilation of $(S_1^* , \dots ,S_{n-1}^*, P^*)$.

The following theorem provides a minimal $\Gamma_n$-isometric dilation to a
certain class of pure $\Gamma_n$-contractions.

\begin{thm}[\cite{sourav6}, Theorem 4.3] \label{dilation-theorem}
Let $(S_1,\dots , S_{n-1},P)$ be a pure $\Gamma_n$-contraction on
a Hilbert space $\mathcal{H}$ and let the fundamental operator
tuple $(F_{1*},\dots,F_{{n-1}*})$ of $(S_1^*,\dots,S_{n-1}^*,P^*)$
be such that $$\left(
\frac{n-1}{n}(F_{1*}^*+F_{{n-1}^*}z),\frac{n-2}{n}(F_{2*}^*+F_{{n-2}^*}z),
\dots, \frac{1}{n}(F_{{n-1}*}^*+F_{1*}z) \right)$$ is a
$\Gamma_{n-1}$-contraction for all $z \in \mathbb T$. Consider the
operators $T_1,\dots,T_{n-1},V$ on $\mathcal{K}=H^2
\otimes \mathcal{D}_{P^*}$ defined by
\[
T_i=I\otimes F_{i*}^*+M_z\otimes F_{{n-i}*}, \text{ for }
i=1,\dots, n-1 \text{ and } V=M_z\otimes I.
\]
Then $(T_1,\dots,T_{n-1},V)$ is a minimal pure
$\Gamma_n$-isometric dilation of $(S_1,\dots,S_{n-1},P)$. Moreover, $(T_1^*,\dots,T_{n-1}^*,V^*)$ is a $\Gamma_n$-co-isometric extension of $(S_1^*,\dots, S_{n-1}^*,P^*)$. 
\end{thm}
A proof to this result can also be found in author's unpublished article \cite{sourav15} (see Theorem 4.2 in \cite{sourav15}).

\subsection{A model theorem for pure $\Gamma_n$-isometry}

It is well-known that by Wold-decomposition an isometry is decomposed into two
orthogonal parts of which one is a unitary and the other is a pure isometry. See Chapter-I of
classic \cite{nagy} for details. A pure isometry $V$ is unitarily
equivalent to the Toeplitz operator $T_z$ on the vectorial Hardy
space $H^2(\mathcal D_{V^*})$. We have in Theorem \ref{G-isometry} an
analogous Wold-decomposition for $\Gamma_n$-isometries in terms of
a $\Gamma_n$-unitary and a pure $\Gamma_n$-isometry. Therefore, a
concrete model for pure $\Gamma_n$-isometries gives a complete
picture of a $\Gamma_n$-isometry. In \cite{BSR}, Biswas and Shyam
Roy produced an abstract model for pure $\Gamma_n$-isometries in terms
of Toeplitz operator tuple $(T_{\phi_1},\dots,
T_{\phi_{n-1}},T_z)$ on a vectorial Hardy-Hilbert space, (\cite{BSR}, Theorem 4.10). Here we make a refinement of their result by
specifying the Hardy-Hilbert space and the corresponding Toeplitz
operators precisely.

\begin{thm}\label{model1}
Let $(\hat{S_1},\dots,\hat{S}_{n-1},\hat{P})$ be a commuting
triple of operators on a Hilbert space $\mathcal H$. If
$(\hat{S_1},\dots,\hat{S}_{n-1},\hat{P})$ is a pure
$\Gamma_n$-isometry then there is a unitary operator $U:\mathcal H
\rightarrow H^2(\mathcal D_{{\hat{P}}^*})$ such that
\[
\hat{S}_i=U^*T_{\varphi_i}U, \text{ for } i=1,\dots,n-1\,, \text{
and }  \hat{P}=U^*T_zU\,.
\]
Here each $T_{\varphi_i}$ is the Toeplitz operator on the
vectorial Hardy space $H^2(\mathcal D_{{\hat{P}}^*})$ with the
symbol $\varphi_i(z)= F_i^*+F_{n-i}z$, where $(F_1,\dots,F_{n-1})$
is the fundamental operator tuple of
$(\hat{S}_1^*,\dots,\hat{S}_{n-1}^*,\hat{P}^*)$ such that
$$\left(
\frac{n-1}{n}(F_1^*+F_{n-1}z),\frac{n-2}{n}(F_2^*+F_{n-2}z),\dots,
\frac{1}{n}(F_{n-1}^*+F_1z) \right)$$ is a
$\Gamma_{n-1}$-contraction for every $z\in \mathbb T$.\\

Conversely, if $F_1,\dots,F_{n-1}$ are bounded operators on a
Hilbert space $E$ such that $$\left(
\frac{n-1}{n}(F_1^*+F_{n-1}z),\frac{n-2}{n}(F_2^*+F_{n-2}z),\dots,
\dfrac{1}{n}(F_{n-1}^*+F_1z) \right)$$ is a
$\Gamma_{n-1}$-contraction for every $z\in \mathbb D$, then
$(T_{F_1^*+F_{n-1}z},\dots,T_{F_{n-1}^*+F_1z},T_z)$ on $H^2(E)$ is
a pure $\Gamma_n$-isometry.
\end{thm}

\begin{proof}

Suppose that $(\hat{S}_1,\dots,\hat{S}_{n-1},\hat{P})$ is a pure
$\Gamma_n$-isometry. Then $\hat{P}$ is a pure isometry and it can
be identified with the Toeplitz operator $T_z$ on $H^2(\mathcal
D_{{\hat{P}}^*})$. Therefore, there is a unitary $U$ from
$\mathcal H$ onto $H^2(\mathcal D_{{\hat{P}}^*})$ such that
$\hat{P}=U^*T_zU$. Since for $\hat{S}_i$ is a commutant of
$\hat{P}$, there are multipliers $\phi_i$ in $H^{\infty}(\mathcal
B(\mathcal D_{{\hat{P}}^*}))$ such that $\hat{S}_i=U^*T_{\phi_i}U$
for $i=1,\dots, n-1$.

Since $(T_{\phi_1},\dots,T_{\phi_{n-1}},T_z)$ is a
$\Gamma_n$-isometry, by part-(2) of Theorem \ref{G-isometry}, we have
that
\[
T_{\phi_i}=T_{\phi_{n-i}}^*T_z \text{ for } i=1,\dots,n-1\,,
\]
and by these relations we have
\[
\phi_i(z)=G_i+G_{n-i}z \textup{ and }
\phi_{n-i}(z)=G_{n-i}^*+G_i^*z \textup{ for some } G_i,G_{n-i}
\in\mathcal B(\mathcal D_{{\hat{P}}^*}).
\]
Set ${F_i}= G_i^*$ and ${F_{n-i}}=G_{n-i}$ for each $i$. Again
since $(T_{\phi_1},\dots,T_{\phi_{n-1}},T_z)$ is a
$\Gamma_n$-isometry, by Lemma 2.7 in \cite{BSR}, $
(\frac{n-1}{n}T_{\phi_1},\dots,\frac{1}{n}T_{\phi_{n-1}}) $ is a
$\Gamma_{n-1}$-contraction and hence
$(\frac{n-1}{n}\phi_1,\dots,\frac{1}{n}\phi_{n-1})$, that is,
\[
\left(
\frac{n-1}{n}(F_1^*+F_{n-1}z),\frac{n-2}{n}(F_2^*+F_{n-2}z),\dots,
\frac{1}{n}(F_{n-1}^*+F_1z)
\right)
\]
is a $\Gamma_{n-1}$-contraction for all $z$ in $\mathbb T$. For the converse part, note that
\[
\left(
\frac{n-1}{n}(F_1^*+F_{n-1}z),\frac{n-2}{n}(F_2^*+F_{n-2}z),\dots,
\frac{1}{n}(F_{n-1}^*+F_1z)
\right)
\]
is a $\Gamma_{n-1}$-contraction. So, the fact that
$(T_{F_1^*+F_{n-1}z},\dots,T_{F_{n-1}^*+F_1z},T_z)$ on $H^2(E)$ is
a $\Gamma_n$-isometry follows from part-(2) of Theorem
\ref{G-isometry}. Needless to mention that $T_z$ on $H^2(E)$ is a
pure isometry which makes
$(T_{F_1^*+F_{n-1}z},\dots,T_{F_{n-1}^*+F_1z},T_z)$ a pure
$\Gamma_n$-isometry.

\end{proof}

\begin{rem}\label{partial-converse}
One can easily verify that the $(n-1)$-tuple $(F_1,\dots, F_{n-1})$ is the $\ft$-tuple
of the adjoint of
$(T_{F_1^*+F_{n-1}z},\dots,T_{F_{n-1}^*+F_1z},T_z)$ as was shown in \cite{sourav6} by the author. This also concludes a partial converse to the
Theorem \ref{existence-uniqueness} in the following sense: given
$n-1$ Hilbert space operators $F_1,\dots, F_{n-1}$ on $E$ such
that
\[
\left(
\frac{n-1}{n}(F_1^*+F_{n-1}z),\frac{n-2}{n}(F_2^*+F_{n-2}z),\dots,
\frac{1}{n}(F_{n-1}^*+F_1z) \right)
\]
is a $\Gamma_{n-1}$-contraction, there exists a
$\Gamma_n$-contraction (which is more precisely a
$\Gamma_n$-co-isometry)
$(T_{F_1^*+F_{n-1}z}^*,\dots,T_{F_{n-1}^*+F_1z}^*,T_z^*)$ on
$H^2(E)$ whose $\ft$-tuple is $(F_1,\dots, F_{n-1})$.

\end{rem}

\begin{lem}\label{lem:normal-op1}
Let $X \subset \C^n$ be polynomially convex and let $(T_1, \dots ,T_n)$ be a commuting tuple of normal Hilbert space operators. Then $X$ is a spectral set for $(T_1, \dots ,T_n)$ if and only if $\sigma_T (T_1, \dots ,T_n) \subseteq X$.
\end{lem}

\begin{proof}
If $X$ is a spectral set for $(T_1, \dots ,T_n)$, then by definition $\sigma_T (T_1, \dots ,T_n) \subseteq X$. So, we prove the other direction. Let $\sigma_T (T_1, \dots ,T_n) \subseteq X$. Since $X$ is polynomially convex, it suffices to prove von Neumann's inequality for polynomials only. Let $f \in \C [z_1, \dots , z_n]$. Then $f(T_1, \dots ,T_n)$ is a normal operator as $(T_1, \dots ,T_n)$ is a commuting normal tuple. Therefore, $\| f(T_1, \dots ,T_n) \| = r \left( f(T_1, \dots ,T_n) \right) $ and by spectral mapping theorem $$\sigma(f(T_1, \dots ,T_n))= \{ f(t_1, \dots, t_n)\,: \, (t_1, \dots, t_n) \in \sigma_T(T_1, \dots ,T_n) \}.$$ Since $\sigma_T(T_1, \dots ,T_n)\subseteq X$, it follows that
\[
\| f(T_1, \dots ,T_n) \|=r (f(T_1, \dots ,T_n)) \leq \sup_{z\in X}\, |f(z)| =\|f\|_{\infty, X}.
\]
This completes the proof.
\end{proof}

\vspace{0.4cm}

Now we are in a position to present the desired dilation theorem. This is another main result of this paper.

\subsection{A matricial von-Neumann inequality and rational dilation on distinguished varieties} \label{sec:07}

\begin{thm}\label{thm:VN}
Let $\Sigma=(S_1,\dots,S_{n-1},P)$ be a $\Gamma_n$-contraction with $\ft$-tuple $(F_1,\dots,F_{n-1})$ such that $P^*$ is pure and $\dim \mathcal D_{P} < \infty $. If $(F_1,\dots,F_{n-1})$ defines a distinguished variety $\Lambda_{\Sigma}$ in $\gn$, then both $(S_1, \dots , S_{n-1},P)$ and $(S_1^*,\dots,S_{n-1}^*,P^*)$ admit normal $\partial \ov{\Lambda}_{\Sigma}-$dilation, where $\partial \ov{\Lambda}_{\Sigma}= \ov{\Lambda}_{\Sigma} \setminus \Lambda_{\Sigma}\;(=b\Gamma_n \cap \ov{\Lambda}_{\Sigma}) $. Moreover, the dilation of $(S_1^*, \dots ,S_{n-1}^*,P^*)$ is minimal and acts on the minimal unitary dilation space of $P^*$.
\end{thm}

\begin{proof}

By hypothesis $\mathcal D_P$ is finite
dimensional. Suppose $\dim \mathcal D_{P}= k $. Then $\mathcal D_P \equiv \C^k$ and $F_1,\dots,F_{n-1}$ are all complex matrices of order $k$. Since $(F_1, \dots ,F_{n-1})$ defines a distinguished variety in $\gn$, it follows from Theorem \ref{thm:DVchar} that the polynomials $\{f_i=\det (F_i^*+pF_{n-i}-s_iI):i=1, \dots, n-1 \}$ form a regular sequence and that
\begin{align*}
\quad \quad \Lambda_{\Sigma} = V_S \cap \gn & =  \{ (s_1,\dots,s_{n-1},p)\in \mathbb G_n \,:
\nonumber
\\& \quad \; (s_1,\dots,s_{n-1}) \in \sigma_T(F_1^*+pF_{n-1}\,,\,
F_2^*+pF_{n-2}\,,\,\dots\,, F_{n-1}^*+pF_1) \}
\end{align*}
is a distinguished variety in $\gn$, where $V_S$ is the complex algebraic variety generated by the polynomials $\{ f_1, \dots , f_{n-1} \}$.

It is well-known from the Sz. Nagy-Foias model theory that unto a unitary the multiplication operator $M_z$ on $L^2(\mathcal{D_P})$ is the minimal unitary dilation of the pure contraction $P^*$. We now show that $(S_1^*, \dots ,S_{n-1}^*,P^*)$ dilates on $\ov{\Lambda}_{\Sigma}$ to the $\Gamma_n$-unitary
$
\left( M_{F_1^*+F_{n-1}z}, \dots ,M_{F_{n-1}+F_1z}, \, M_z   \right)
$
acting on $L^2(\mathcal D_P)=L^2(\C^k)$. Needless to mention that this dilation must be a minimal one. It follows from the condition-(1) of Theorem \ref{thm:DVchar} that $(F_1^*+zF_{n-1}, F_2^*+zF_{n-2},\dots,F_{n-1}^*+zF_1, zI )$ is a commuting normal tuple for each $z\in \T$. Since $(F_1, \dots ,F_{n-1})$ defines a distinguished variety in $\gn$, then it follows from the proof of Theorem \ref{thm:DVchar} (see the paragraph after equation (\ref{eqn:essn})) that the operator tuple $\left( T_{F_1^*+F_{n-1}z}^*, \dots, T_{F_{n-1}^*+F_{1}z}^* , T_z^* \right)$ acting on $H^2(\C^k)$ is a $\Gamma_n$-co-isometry for every $z\in \T$. A simple calculation shows that $(F_1, \dots ,F_{n-1})$ is the $\ft$-tuple of $\left( T_{F_1^*+F_{n-1}z}^*, \dots, T_{F_{n-1}^*+F_{1}z}^* , T_z^* \right)$. It follows from Lemma \ref{lem:BSR2} that
\[
\left(
\frac{n-1}{n}(F_1^*+F_{n-1}z),\frac{n-2}{n}(F_2^*+F_{n-2}z),\dots, \frac{1}{n}(F_{n-1}^*+F_1z) \right)
\]
is a $\Gamma_{n-1}$-contraction for all $z\in \T$. Note that $\|F_i^*+zF_{n-i} \| \leq {n \choose i}$ by Theorem \ref{existence-uniqueness} as $(F_1, \dots ,F_{n-1})$ is the $\ft$-tuple. Then, it follows from Theorem \ref{G-unitary} that the multiplication operator tuple
$
\left( M_{F_1^*+F_{n-1}z}, \dots ,M_{F_{n-1}+F_1z}, M_z \right)$
acting on $L^2(\C^k)$ is a $\Gamma_n$-unitary as it also satisfies $M_{F_i^*+zF_{n-i}}=M_{F_{n-i}^*+zF_i}^*M_z$, and $M_z$ is a unitary. Therefore, for each $z\in \D$ the Taylor joint spectrum $\sigma_T(F_1^*+zF_{n-1}, \dots, F_{n-1}^*+zF_1, zI) \subseteq b\Gamma_n$. Again since $F_1, \dots, F_{n-1}$ defines a distinguished variety in $\gn$, it follows from Theorem \ref{thm:DVchar} that
\[
\sigma_T(F_1^*+zF_{n-1}, \dots, F_{n-1}^*+zF_1, zI) \subseteq b\Gamma_n \cap V_S =\partial \ov{\Lambda}_{\Sigma}.
\]
Therefore, for every $z\in \T$, $(F_1^*+zF_{n-1}, \dots, F_{n-1}^*+zF_1, zI)$ is a commuting normal tuple whose Taylor joint spectrum lies on $\partial \ov{\Lambda}_{\Sigma}\subseteq b\Gamma_n$. Thus, $(F_1^*+zF_{n-1}, \dots, F_{n-1}^*+zF_1, zI)$ and hence $\left( M_{F_1^*+F_{n-1}z}, \dots ,M_{F_{n-1}+F_1z}, M_z \right)
$ is a $\Gamma_n$-unitary and by Lemma \ref{lem:normal-op1}, $\partial \ov{\Lambda}_{\Sigma}$ is a spectral set for it. Again, since for every $z\in \T$,
\[
\left(
\frac{n-1}{n}(F_1^*+F_{n-1}z),\frac{n-2}{n}(F_2^*+F_{n-2}z),\dots, \frac{1}{n}(F_{n-1}^*+F_1z) \right)
\]
is a $\Gamma_{n-1}$-contraction, we apply Theorem \ref{dilation-theorem} to the pure $\Gamma_n$-contraction
$(S_1^*,\dots,S_{n-1}^*,P^*)$ to get a $\Gamma_n$-co-isometric extension $(T_{\varphi_1}^*,\dots
,T_{\varphi_{n-1}}^*,T_z^*)$ on $H^2(\mathcal D_{P})$ of $(S_1,\dots,S_{n-1},P)$, where $\varphi_i(z)=F_i^*+F_{n-i}z$ for $i=1,\dots,n-1$.
Therefore, we have that
\[
T_{F_1^*+F_{n-1}z}^*|_{\mathcal H}=S_1, \dots
,T_{F_{n-1}^*+F_1z}^*|_{\mathcal H}=S_{n-1}, \text{ and }
T_z^*|_{\mathcal H}=P.
\]
Therefore, $(T_{\varphi_1}, \dots , T_{\varphi_{n-1}}, T_z)$ is a $\Gamma_n$-isometric dilation of $(S_1^*, \dots ,S_{n-1}^*, P^*)$. Also, the restriction of the $\Gamma_n$-unitary $(M_{\varphi_1},\dots,M_{\varphi_{n-1}},M_z)$ to
the joint-invariant subspace $H^2(\mathcal D_P)$ is the $\Gamma_n$-isometry $(T_{\varphi_1},\dots,
T_{\varphi_{n-1}},T_z)$. Thus, $(M_{\varphi_1},\dots,M_{\varphi_{n-1}},M_z)$ on
$L^2(\mathcal D_P)$ is a $\Gamma_n$-unitary dilation of $(S_1^*, \dots , S_{n-1}^*,P^*)$, where $L^2(\mathcal D_P)$ is the Sz. Nagy-Foias minimal unitary dilation space for the pure contraction $P^*$. Also, $\partial \ov{\Lambda}_{\Sigma}$ is a spectral set $(M_{\varphi_1},\dots,M_{\varphi_{n-1}},M_z)$. So, we are done with the first part of the proof.\\

We now show that $(S_1, \dots ,S_{n-1},P)$ also dilates to a $\Gamma_n$-unitary having $\partial \ov{\Lambda}_{\Sigma}$ as a spectral set. By Theorem \ref{Arveson}, it suffices if we prove that $\ov{\Lambda}_{\Sigma}$ is a complete spectral set for $(S_1, \dots ,S_{n-1},P)$. Let $f$ be a matrix-valued holomorphic polynomial in
$n$-variables, where the coefficient matrices are of order $d$ say, and
let $f_*$ be the polynomial satisfying
$f_*(A_1,\dots,A_n)=f(A_1^*,\dots,A_{n}^*)^*$ for any $n$
commuting operators $A_1,\dots,A_n$. Then
\begingroup
\allowdisplaybreaks
\begin{align*}
\|f(S_1,\dots,S_{n-1},P)\| &\leq \|f(T_{\varphi_1}^*,\dots,T_{\varphi_{n-1}}^*, T_z^*)\|_{H^2( \mathcal D_{P})\otimes \mathbb C^d} \\
&= \|f_*(T_{\varphi_1},\dots,T_{\varphi_{n-1}}, T_z)^*\|_{H^2( \mathcal D_{P})\otimes \mathbb C^d} \\
& = \|f_*(T_{\varphi_1},\dots,T_{\varphi_{n-1}},
T_z)\|_{H^2(\mathcal D_{P})\otimes \mathbb C^d} \\&
\leq \|f_*(M_{\varphi_1},\dots,M_{\varphi_{n-1}},M_z)\|_{L^2( \mathcal D_{P})\otimes \mathbb C^d} \\
& = \max_{\theta \in [0,2\pi]} \|f_*
(\varphi_1(e^{i\theta}),\dots,\varphi_{n-1}(e^{i\theta}),
e^{i\theta}I)\|.
\end{align*}
\endgroup
Since $M_{\varphi_1},\dots,
M_{\varphi_{n-1}},M_z $ are commuting normal operators,
so are $\varphi_1(z), \dots, \varphi_{n-1}(z)$ for every $z\in\mathbb T$. Therefore,
\begingroup
\allowdisplaybreaks
\begin{align*}
\|f_*(\varphi_1(e^{i\theta}),\dots,\varphi_{n-1}(e^{i\theta}),
e^{i\theta}I)\| =  \max_{\theta} & \{
|f_*(\lambda_1,\dots,\lambda_{n-1},e^{i\theta})|: \\&
 (\lambda_1,\dots,\lambda_{n-1}, e^{i\theta})\in
\sigma_T(\varphi(e^{i\theta}),\dots,\varphi_{n-1}(e^{i\theta}), e^{i\theta}I) \}.
\end{align*}
\endgroup
Let us define
\begingroup
\allowdisplaybreaks
\begin{align*}
\Lambda_{\Sigma}^*= &\{ (s_1,\dots,s_{n-1},p)\in \mathbb G_n \,:
\nonumber
\\& \quad (s_1,\dots,s_{n-1}) \in \sigma_T(F_1+pF_{n-1}^*,F_2+pF_{n-2}^*,\dots,F_{n-1}+pF_1^*)
\}.
\end{align*}
\endgroup
Then $F_1^*,\dots, F_{n-1}^*$ also satisfy the conditions of Theorem \ref{thm:DVchar} with $\Lambda_{\Sigma}$ being replaced by $\Lambda_{\Sigma}^*$. So, we can conclude that $\Lambda_\Sigma^*$ is also a
distinguished variety in $\mathbb G_n$. Also, since $(M_{\varphi_1},\dots,M_{\varphi_{n-1}},M_z)$ is a $\Gamma$-unitary, each $(\lambda_1,\dots,\lambda_{n-1},e^{i\theta})$ in $\sigma_T(\varphi_1(e^{i\theta}),\dots,\varphi_{n-1}(e^{i\theta}),e^{i\theta}I)$ belongs to $b\Gamma_n$. Therefore, we have that
\begingroup
\allowdisplaybreaks
\begin{align*}
\| f(S_1,\dots, S_{n-1},P) \| &\leq \max_{(s_1,\dots,
s_{n-1},p)\in \overline{\Lambda}_{\Sigma^*}
\cap b\Gamma_n} \| f_*(s_1,\dots,s_{n-1},p) \| \\
& =\max_{(s_1,\dots,s_{n-1},p)\in \overline{ \Lambda}_{\Sigma^*}
\cap b\Gamma_n} \|
\overline{f(\bar{s_1},\dots,\bar{s}_2,\bar{p})} \| \\
& = \max_{(s_1,\dots,s_{n-1},p)\in \overline{\Lambda}_{\Sigma^*}
\cap b\Gamma_n} \| f(\bar{s}_1,\dots,\bar{s}_2,\bar{p}) \|\\& =\max_{(s_1,\dots,s_{n-1},p)\in \overline{\Lambda}_{\Sigma} \cap b\Gamma_n} \|
f(s_1,\dots,s_{n-1},p) \| \\
&  \leq \max_{(s_1,\dots,s_{n-1},p)\in
\overline{\Lambda}_{\Sigma}} \| f(s_1,\dots,s_{n-1},p) \|.
\end{align*}
\endgroup

Hence $\ov{\Lambda}_{\Sigma}$ is a complete spectral set for $(S_1, \dots ,S_{n-1},P)$ and the proof is complete.

\end{proof}

\vspace{0.2cm}

\noindent \textbf{Concluding remarks.} Our curiosity extends to ask if the converse to Theorem \ref{thm:VN} is true, i.e. if a $\Gamma_n$-contraction $(S_1, \dots ,S_{n-1},P)$ possesses such a minimal normal $\partial\ov{\Lambda}_{\Sigma}-$dilation, then does it follow that $\Lambda_{\Sigma}$ is determined by the $\ft$-tuple of $(S_1, \dots ,S_{n-1}, P)$ ? There will be a sequel of this paper, where we shall show that the answer is affirmative. This will be achieved by exhibiting an explicit and unique minimal rational dilation on $\Gamma_n$ and numerous operator theoretic and complex analytic results will be proved to establish that.\\

\vspace{0.4cm}

\noindent \textbf{Conflict of interest.} There is no conflict of interest.

\section{Data availability statement}

\begin{enumerate}
\item Data sharing is not applicable to this article as no datasets were generated or analysed
during the current study.

\item In case any datasets are generated during and/or analysed during the current study, they
must be available from the corresponding author on reasonable request.
\end{enumerate}

\end{document}